\documentclass[11pt,leqno]{article}
\usepackage{hyperref}
\usepackage{soul}
\usepackage[doc]{optional}
\usepackage{xcolor}
\definecolor{labelkey}{rgb}{0,0.08,0.45}
\definecolor{rekey}{rgb}{0,0.6,0.0}
\definecolor{Brown}{rgb}{0.45,0.0,0.05}
\usepackage{exscale,relsize}
\usepackage{amsmath}
\usepackage{amsfonts}
\usepackage{amssymb}
\usepackage{calc}
\usepackage{theorem}
\usepackage{pifont}      
\usepackage{graphicx}
\DeclareMathOperator{\weakstarly}{\rightharpoondown_{\mathrm{w*}}}

\DeclareMathOperator{\weakstarlynorm}{\stackrel{\mathrm{w{**}}\times\|\cdot\|}{\rightharpoondown}}

\newcommand{\wk}{\ensuremath{\operatorname{w*}}}

\newcommand{\scal}[2]{\langle{{#1},{#2}}\rangle}

\newcommand{\RR}{\ensuremath{\mathbb R}}

\newcommand{\RX}{\ensuremath{\,\left]-\infty,+\infty\right]}}
\newcommand{\RXX}{\ensuremath{\,\left[-\infty,+\infty\right]}}

\newcommand{\NN}{\ensuremath{\mathbb N}}

\newcommand{\thalb}{\ensuremath{\tfrac{1}{2}}}

\newcommand{\menge}[2]{\big\{{#1} \mid {#2}\big\}}

\newcommand{\To}{\ensuremath{\rightrightarrows}}

\newcommand{\spand}{\operatorname{span}}

\newcommand{\pos}{\operatorname{pos}}
\newcommand{\J}{\ensuremath{\mathbf{J}}}

\newcommand{\dom}{\ensuremath{\operatorname{dom}}}

\newcommand{\gra}{\ensuremath{\operatorname{gra}}}
\newcommand{\epi}{\ensuremath{\operatorname{epi}}}

\newcommand{\inte}{\ensuremath{\operatorname{int}}}

\newcommand{\bd}{\ensuremath{\operatorname{bdry}}}
\newcommand{\ran}{\ensuremath{\operatorname{ran}}}

\newcommand{\conv}{\ensuremath{\operatorname{conv}}}

\newcommand{\Id}{\ensuremath{\operatorname{Id}}}

\newcommand{\pinf}{\ensuremath{+\infty}}

\renewcommand{\phi}{\ensuremath{\varphi}}

\newcommand{\eph}{\operatorname{epi}}

\newtheorem{theorem}{Theorem}[section]
\newtheorem{lemma}[theorem]{Lemma}
\newtheorem{fact}[theorem]{Fact}
\newtheorem{corollary}[theorem]{Corollary}
\newtheorem{proposition}[theorem]{Proposition}
\newtheorem{definition}[theorem]{Definition}

\newtheorem{openprob}[theorem]{Open Problem}
\theoremstyle{plain}{\theorembodyfont{\rmfamily}
}
\theoremstyle{plain}{\theorembodyfont{\rmfamily}
}
\theoremstyle{plain}{\theorembodyfont{\rmfamily}
}
\theoremstyle{plain}{\theorembodyfont{\rmfamily}
\newtheorem{example}[theorem]{Example}}
\theoremstyle{plain}{\theorembodyfont{\rmfamily}
\newtheorem{remark}[theorem]{Remark}}
\newtheorem{problem}[theorem]{Open Problem}
\theoremstyle{plain}{\theorembodyfont{\rmfamily}
}

\newcommand{\qede}{\hspace*{\fill}$\Diamond$\medskip}
\begin{document}


\title{\sffamily{Recent progress on  Monotone Operator Theory}}

\author{
Jonathan M. Borwein\thanks{CARMA, University of Newcastle,
 Newcastle, New South Wales 2308, Australia. E-mail:
\texttt{jonathan.borwein@newcastle.edu.au}. Laureate Professor at the University of Newcastle and Distinguished Professor at  King
Abdul-Aziz University, Jeddah.}\;
  and Liangjin\
Yao\thanks{CARMA, University of Newcastle,
 Newcastle, New South Wales 2308, Australia.
E-mail:  \texttt{liangjin.yao@newcastle.edu.au}.}}

\date{May 21, 2013}
\maketitle

\begin{abstract} \noindent
In this paper, we  survey recent progress on the theory of maximally monotone operators  in general Banach space.
 We also extend several results and leave some open questions.
\end{abstract}



\section{Introduction}

We assume throughout that
$X$ is a real Banach space with norm $\|\cdot\|$,
that $X^*$ is the continuous dual of $X$,
 and
that $X$ and $X^*$ are paired by $\scal{\cdot}{\cdot}$.
The \emph{open unit ball} and \emph{closed unit ball} in $X$ is denoted respectively by $U_X:=
\menge{x\in X}{\|x\|<1}$ and  $B_X:=
\menge{x\in X}{\|x\|\leq1}$, and $\NN:=\{1,2,3,\ldots\}$.

We recall the following basic
fact regarding the second dual ball:

\begin{fact}[Goldstine]\emph{(See \cite[Theorem~2.6.26]{Megg} or \cite[Theorem~3.27]{FabianHH}.)}
  \label{Goldst:1}
 The weak*-closure of $B_X$ in $X^{**}$ is $B_{X^{**}}$.
\end{fact}

We say  a net $(a_{\alpha})_{\alpha\in\Gamma}$ in X is \emph{eventually bounded}
 if there exist $\alpha_0\in\Gamma$ and $M\geq0$ such that
\begin{align*}
\|a_{\alpha}\|\leq M,\quad \forall \alpha\succeq_\Gamma\alpha_0.
\end{align*}
We denote by $\longrightarrow$
and $\weakstarly$
the norm convergence and weak$^*$ convergence
of  nets,  respectively.

\subsection{Monotone operators}

Let $A\colon X\To X^*$
be a \emph{set-valued operator} (also known as a relation, point-to-set mapping or multifunction)
from $X$ to $X^*$, i.e., for every $x\in X$, $Ax\subseteq X^*$,
and let
$\gra A := \menge{(x,x^*)\in X\times X^*}{x^*\in Ax}$ be
the \emph{graph} of $A$. The \emph{domain} of $A$
  is $\dom A:= \menge{x\in X}{Ax\neq\varnothing}$ and
$\ran A:=A(X)$ is the \emph{range} of $A$.

Recall that $A$ is
\emph{monotone} if
\begin{equation}
\scal{x-y}{x^*-y^*}\geq 0,\quad \forall (x,x^*)\in \gra A\;
\forall (y,y^*)\in\gra A,
\end{equation}
and \emph{maximally monotone} if $A$ is monotone and $A$ has
 no proper monotone extension
(in the sense of graph inclusion).
Let $A:X\rightrightarrows X^*$ be monotone and $(x,x^*)\in X\times X^*$.
 We say $(x,x^*)$ is \emph{monotonically related to}
$\gra A$ if
\begin{align*}
\langle x-y,x^*-y^*\rangle\geq0,\quad \forall (y,y^*)\in\gra
A.\end{align*}
Monotone operators have frequently shown themselves to be a key class of objects in both
modern Optimization and Analysis; see, e.g., \cite{Bor1,Bor2,Bor3},
the books \cite{BC2011,
BorVan,BurIus,ph,Si,Si2,Rock70CA,RockWets,Zalinescu,Zeidler2A, Zeidler2B}
and the references given therein.

We now introduce the four fundamental properties of maximally
monotone operators that our paper focusses on.

 \begin{definition}\label{def1}
 Let $A:X\To X^*$ be maximally monotone.
 Then four key properties of monotone operators are defined as follows.
 \begin{enumerate}
 \item $A$ is
\emph{of dense type or type (D)} (1971, \cite{Gossez3}, \cite{ph2} and \cite[Theorem~9.5]{Si5}) if for every
$(x^{**},x^*)\in X^{**}\times X^*$ with
\begin{align*}
\inf_{(a,a^*)\in\gra A}\langle a-x^{**}, a^*-x^*\rangle\geq 0,
\end{align*}
there exist a  bounded net
$(a_{\alpha}, a^*_{\alpha})_{\alpha\in\Gamma}$ in $\gra A$
such that
$(a_{\alpha}, a^*_{\alpha})_{\alpha\in\Gamma}$
weak*$\times$strong converges to
$(x^{**},x^*)$.
\item $A$ is
\emph{of type negative infimum (NI)} (1996, \cite{SiNI}) if
\begin{align*}
\inf_{(a,a^*)\in\gra A}\langle a-x^{**},a^*-x^*\rangle\leq0,
\quad \forall(x^{**},x^*)\in X^{**}\times X^*.
\end{align*}

\item
$A$ is \emph{of type Fitzpatrick-Phelps (FP)} (1992, \cite{FP}) if
whenever $U$ is an open convex subset of $X^*$ such that $U\cap \ran
A\neq\varnothing$, $x^*\in U$, and $(x,x^*)\in X\times X^*$ is
monotonically related to $\gra A\cap (X\times U)$ it must follow
that $(x,x^*)\in\gra A$.

\item $A$ is
\emph{of  ``Br{\o}nsted-Rockafellar" (BR) type  (1999, \cite{Si6})
if whenever $(x,x^*)\in X\times X^*$, $\alpha,\beta>0$ and
\begin{align*}\inf_{(a,a^*)\in\gra A} \langle x-a,x^*-a^*\rangle
>-\alpha\beta\end{align*} then there exists $(b,b^*)\in\gra A$ such
that $\|x-b\|<\alpha,\|x^*-b^*\|<\beta$.}
\end{enumerate}
\end{definition}
As is now known (see Corollary~\ref{cor:main} and \cite{Si4, SiNI, MarSva}),
the first three properties coincide.
This coincidence is central to many of our proofs. Fact~\ref{MAS:BR1}
 also shows us that every maximally monotone operator of type (D) is of
type (BR).
(The converse fails, see Example~\ref{FPEX:1}\ref{BCCE:A7}.)
Moreover, in reflexive space every maximally monotone operator is
of type (D), as is the subdifferential operator of every proper closed convex function
on a Banach space.

While monotone operator theory is rather complete
in reflexive space
--- and for type (D) operators in general space --- the general
situation is less clear \cite{BorVan,Bor3}. Hence our continuing
interest in operators which are not of type (D). Not  every maximally monotone operator is
of type (BR) (see Example~\ref{FPEX:1}\ref{BCCE:Au2}).

We  say  a Banach space $X$ is \emph{of type (D)} \cite{Bor3}
if every maximally monotone operator  on $X$ is of  type (D). At
present the only known type (D) spaces are the reflexive spaces; and
our work here suggests that there are no non-reflexive type (D)
spaces.  In \cite[Exercise 9.6.3, page~450]{BorVan} such spaces were called
(NI) spaces and some potential non-reflexive examples were
conjectured; all of which are ruled out by our more recent work. In
\cite[Theorem 9.7.9, page~458]{BorVan} a variety of the pleasant properties of
type (D) spaces was listed.  In Section~\ref{s:DFPV} we  briefly study a new  dual class of (DV) spaces.

\subsection{Convex analysis}

As much as  possible we adopt standard convex analysis notation.
Given a subset $C$ of $X$,
$\inte C$ is the \emph{interior} of $C$ and
$\overline{C}$ is   the
\emph{norm closure} of $C$.
 For the set $D\subseteq X^*$, $\overline{D}^{\wk}$
is the weak$^{*}$ closure of $D$, and  the norm $\times$ weak$^*$ closure of $C\times D$ is
$\overline{C\times D}^{\|\cdot\|\times\wk}$.
The \emph{indicator function} of $C$, written as $\iota_C$, is defined
at $x\in X$ by
\begin{align}
\iota_C (x):=\begin{cases}0,\,&\text{if $x\in C$;}\\
+\infty,\,&\text{otherwise}.\end{cases}\end{align}
The \emph{support function} of $C$, written as $\sigma_C$,
 is defined by $\sigma_C(x^*):=\sup_{c\in C}\langle c,x^*\rangle$.
 For every $x\in
X$, the \emph{normal cone} operator of $C$ at $x$ is defined by
$N_C(x):= \menge{x^*\in X^*}{\sup_{c\in C}\scal{c-x}{x^*}\leq 0}$, if
$x\in C$; and $N_C(x):=\varnothing$, if $x\notin C$;
 the \emph{tangent cone} operator of $C$ at $x$ is defined by
$T_C(x):= \menge{x\in X}{\sup_{x^*\in N_C(x)}\langle x, x^*\rangle\leq 0}$, if
$x\in C$; and $T_C(x):=\varnothing$, if $x\notin C$. The \emph{hypertangent cone} of $C$ at $x$, $H_C(x)$,
coincides with the interior of $T_C(x)$ (see \cite{BorStr,BorStr1}).

Let $f\colon X\to \RX$. Then
$\dom f:= f^{-1}(\RR)$ is the \emph{domain} of $f$, and
$f^*\colon X^*\to\RXX\colon x^*\mapsto
\sup_{x\in X}(\scal{x}{x^*}-f(x))$ is
the \emph{Fenchel conjugate} of $f$.
The \emph{epigraph} of $f$ is $\epi f := \menge{(x,r)\in
X\times\RR}{f(x)\leq r}$.
Let the net $(y_{\alpha}, y_{\alpha}^*)_{\alpha\in I}$ be in $X\times X^*$ and $(x^{**}, x^*)\in X^{**}\times X^*$.
We write $(y_{\alpha}, y_{\alpha}^*)\weakstarlynorm(x^{**}, x^*)$
 when $(y_{\alpha}, y_{\alpha}^*)$ converges to $(x^{**}, x^*)$ in the weak$^{*}$-topology $\omega (X^{**}, X^*)\times\|\cdot\|$.
We say $f$ is proper if $\dom f\neq\varnothing$.
Let $f$ be proper. The \emph{subdifferential} of
$f$ is defined by
   $$\partial f\colon X\To X^*\colon
   x\mapsto \{x^*\in X^*\mid(\forall y\in
X)\; \scal{y-x}{x^*} + f(x)\leq f(y)\}.$$
We  denote  by $J$ \emph{the duality map}, i.e.,
the subdifferential of the function $\tfrac{1}{2}\|\cdot\|^2$
mapping $X$ to $X^*$.
Let $g\colon X\rightarrow\RX$.
Then the \emph{inf-convolution} $f\Box g$
is the function defined on $X$ by
\begin{equation*}f\Box g\colon
x \mapsto \inf_{y\in X}
\big(f(y)+g(x-y)\big).
\end{equation*}
Let $Y$ be another real Banach space and  $F_1, F_2\colon X\times Y\rightarrow\RX$.
Then the \emph{partial inf-convolution} $F_1\Box_1 F_2$
is the function defined on $X\times Y$ by
\begin{equation*}F_1\Box_1 F_2\colon
(x,y)\mapsto \inf_{u\in X}
\big(F_1(u,y)+F_2(x-u,y)\big).
\end{equation*}
Then $F_1\Box_2 F_2$
is the function defined on $X\times Y$ by
\begin{equation*}F_1\Box_2 F_2\colon
(x,y)\mapsto \inf_{v\in Y}
\big(F_1(x,y-v)+F_2(x,v)\big).
\end{equation*}

%

\subsection{Structure of the paper}

The remainder of this paper is organized as follows.
In Section~\ref{s:TypeD},  we construct maximally monotone
 operators that are not of Gossez's dense-type (D) in many nonreflexive
 spaces, and present many related examples such as operators not of type (BR).

 In Section~\ref{s:Thrmon}, we show that monotonicity of dense type (type (D)), negative infimum type and Fitzpatrick-
Phelps type all coincide. We reprise  two recent proofs---by  Marques Alves/Svaiter  and Simons---showing the important result that every maximally monotone operator of negative infimum type defined on a real Banach space is actually of dense type.


In Section~\ref{s:main},
we consider the structure of maximally monotone operators in Banach space whose domains
have nonempty interior and we present new and explicit structure formulas for such operators.
 In Section~\ref{s:linear}, we list some important recent characterizations of monotone linear relations, such as a complete generalization of the Brezis-Browder theorem in general Banach space.
 Finally,  in Section~\ref{s:openp},  we mention some  central open problems in Monotone Operator Theory.

 \section{Type (D) space}\label{s:TypeD}
 In this section,  we construct maximally monotone
 operators that are not of Gossez's dense-type (D) in nearly all nonreflexive
 spaces. Many of these operators also fail to possess the
 Br{\o}nsted-Rockafellar (BR) property.
  Using  these operators,
we show that the partial inf-convolution of two BC--functions will
not always be a BC--function. This provides a negative answer to a
challenging question posed by Stephen Simons. Among other consequences, we
deduce --- in a uniform fashion --- that every Banach space which contains an isomorphic copy of
the James space $\J$ or its dual $\J^*$,  or  $c_0$ or
its dual  $\ell^1$, admits a non type (D) operator.
The existence of non type (D) operators in spaces containing $\ell^1$ or $c_0$
has been proved recently by Bueno and Svaiter \cite{BuSv}.

 This section is based on the work in \cite{BBWY3} by Bauschke, Borwein, Wang and Yao.

 Let $A:X\rightrightarrows X^*$ be linear relation.
  We say that $A$ is
\emph{skew} if $\gra A \subseteq \gra (-A^*)$;
equivalently, if $\langle x,x^*\rangle=0,\; \forall (x,x^*)\in\gra A$.
Furthermore,
$A$ is \emph{symmetric} if $\gra A
\subseteq\gra A^*$; equivalently, if $\scal{x}{y^*}=\scal{y}{x^*}$,
$\forall (x,x^*),(y,y^*)\in\gra A$.
We define the \emph{symmetric part} and the \emph{skew part} of $A$ via
\begin{equation}
\label{Fee:1}
P := \thalb A + \thalb A^* \quad\text{and}\quad
S:= \thalb A - \thalb A^*,
\end{equation}
respectively. It is easy to check that $P$ is symmetric and that $S$
is skew. Let $S$ be a
subspace of $X$.
  We say $A$ is \emph{$S$--saturated} \cite{Si2} if
\begin{align*}
Ax+S^{\bot}=Ax,\quad
\forall x\in\dom A.
\end{align*}
We say a maximally monotone operator $A:X\rightrightarrows X^*$ is
\emph{unique} if all maximally monotone extensions of
$A$ (in the sense of graph inclusion) in $X^{**}\times X^*$ coincide.
Let $Y$ be another real Banach space. We set  $P_X: X\times Y\rightarrow
X\colon (x,y)\mapsto x$,
 and
 $P_Y: X\times Y\rightarrow Y\colon (x,y)\mapsto y$.
Let $L:X\rightarrow Y$ be linear. We say $L$ is a (linear)
\emph{isomorphism} into $Y$ if $L$ is one to one, continuous and
$L^{-1}$ is continuous on $\ran L$. We say $L$ is an \emph{isometry}
if $\|Lx\|=\|x\|,  \forall x\in X$. The spaces $X$, $Y$ are then
\emph{isometric} (\emph{isomorphic}) if there exists an isometry
(\emph{isomorphism})  from $X$ onto $Y$.

 Now let $F:X\times X^*\rightarrow\RX$.
 We say $F$ is a \emph{BC--function} (BC stands for
``Bigger conjugate'') \cite{Si2} if $F$ is proper and
convex with
\begin{align} F^*(x^*,x)
 \geq F(x,x^*)\geq\langle x,x^*\rangle\quad\forall(x,x^*)\in X\times X^*.
 \end{align}

\subsection{Operators of type (BR)}

We first describe some properties of type (BR) operators.
Let $A:X\To X^*$ be a maximally monotone operator. We say $A$ is
\emph{isomorphically of type (BR)}, or \emph{(BRI)} if, $A$ is of  type (BR) in every equivalent norm on $X$.
Let us emphasize that we do not know if there exists a maximally monotone operator of type (BR) that is  not isomorphically of type (BR). Note that all the other properties studied in this paper are preserved by Banach space isomorphism.

To produce operators not of type (D) but that are of type (BR) we
exploit:

\begin{lemma}\emph{(See \cite[Lemma~3.2]{BBWY3}.)}\label{BRSk:1}
Let $A:X\rightrightarrows X^*$ be a maximally monotone and  linear skew operator.
Assume that $\gra(-A^*)\cap X\times X^*\subseteq\gra A$.
Then $A$ is isomorphically of type (BR).
\end{lemma}

Lemma~\ref{BRSk:1} shows that every continuous monotone linear and  skew operator is of type (BR).

\begin{corollary}\emph{(See \cite[Corollary~3.3]{BBWY3}.)}
Let $A:X\rightrightarrows X^*$ be a maximally monotone and  linear
skew operator that  is not of type (D). Assume that $A$ is unique.
Then $\gra A=\gra(-A^*)\cap X\times X^*$  and so $A$ is isomorphically of type
(BR).
\end{corollary}

 \begin{fact}[Marques Alves and Svaiter]\emph{(See
\cite[Theorem~1.4(4)]{MarSva2} or \cite{MarSva3}.)} \label{MAS:BR1} Let $A:X \To X^*$
be a maximally  monotone operator that
is of type (NI) (or equivalently, by Theorem~\ref{retyD:3}, of type (D)).
Then $A$ is isomorphically of type (BR).
\end{fact}

\begin{remark}
Since (NI) is an isomorphic notion, by Fact~\ref{MAS:BR1}, every operator of type (NI)
 is isomorphically of type (BR).
\end{remark}

The next result will allow us to show that not every continuous monotone linear operator is of type (BR) (see Remark~\ref{RProBR:2} below).

\begin{proposition}\label{ProBR1}
Let $A \colon X\rightrightarrows X^*$ be maximally monotone. Assume that
there exists $e\in X^*$
such that
$e \notin \overline{\ran A}$ and
 that
\begin{align*} \langle a^*,a \rangle \ge \langle e,a \rangle^2,\quad \forall
(a,a^*)\in \gra A.
\end{align*}
Then  $A$ is not of type (BR), and $P_{X^*}\left[\dom F_A\right]\nsubseteq\overline{\ran A}$.
\end{proposition}
\begin{proof}
Let $(x_0, x_0^*):=(0,e)$. Then we have
\begin{align}
&\inf_{(a,a^*)\in\gra A}\big(\langle a-x_0, a^*- x^*_0\rangle\big)=
\inf_{(a,a^*)\in\gra A}\big(\langle a, a^*- e\rangle\big)=\inf_{(a,a^*)\in\gra A}\big(\langle a, a^*\rangle-\langle a, e\rangle\big)\nonumber\\
&\geq \inf_{(a,a^*)\in\gra A}\big(\langle a, e\rangle^2-\langle a, e\rangle\big)
\geq \inf_{t\in\RR}\big(t^2-t\big)=-\frac{1}{4}.\label{ProBR1L:1}
\end{align}
Suppose to the contrary that $A$ is of type (BR).
Then
 Fact~\ref{ProBR2} implies that $e\in\overline{\ran A}$, which contradicts the assumption that $e\notin\overline{\ran A}$.
Hence $A$ is not of type (BR).  By \eqref{ProBR1L:1}, $(0,e)\in\dom F_A$ and  $e\notin\overline{\ran A}$. Hence
$P_{X^*}\left[\dom F_A\right]\nsubseteq\overline{\ran A}$.
\end{proof}

\subsection{Operators of type (D)}

We now turn to type (D) operators.

 \begin{fact}[Simons]\label{Satu:1}\emph{(See \cite[Theorem~28.9]{Si2}.)}
 Let $Y$ be a Banach space, and $L:Y\rightarrow X$ be continuous and linear with $\ran L$ closed and
  $\ran L^*=Y^*$.
 Let $A:X\rightrightarrows X^*$ be  monotone
 with $\dom A\subseteq\ran L$ such that $\gra A\neq\varnothing$.
  Then $A$ is maximally monotone
 if, and only if $A$ is $\ran L$--saturated and $L^*AL$ is maximally monotone.
\end{fact}

Fact~\ref{Satu:1}  leads us to  the following result.

\begin{theorem}\label{Simonco:1}\emph{(See \cite[Theorem~2.17]{BBWY3}.)}
 Let $Y$ be a Banach space, and $L:Y\rightarrow X$ be an isomorphism into $X$.
 Let $T:Y\rightrightarrows Y^*$ be monotone.
  Then $T$ is maximally monotone
 if, and only if  $(L^*)^{-1}TL^{-1}$, mapping $X$ into $X^*$, is maximally monotone.
\end{theorem}

The following consequence will allow us to construct maximally
monotone operators that are not of type (D) in a very wide variety of
non-reflexive Banach spaces.

\begin{corollary}[Subspaces]\label{Simonco:2}\emph{(See \cite[Corollary~2.18]{BBWY3}.)}
 Let $Y$ be a Banach space, and $L:Y\rightarrow X$ be an isomorphism into $X$.
  Let $T:Y\rightrightarrows Y^*$ be monotone. The following hold.
  \begin{enumerate}
\item
 \label{STV:s1}
 Assume that $(L^*)^{-1}TL^{-1}$ is maximally monotone of type (D). Then $T$ is maximally monotone of type (D).
 In particular, every Banach subspace of a type (D)
 space is of type (D).
 \item  \label{STV:s2}
If  $T$ is maximally monotone and not of type (D),
then  $(L^*)^{-1}TL^{-1}$ is a maximally monotone operator mapping
$X$ into $X^*$ that
 is not of type (D).
   \end{enumerate}
\end{corollary}

\begin{remark}
Note that it follows that $X$ is of type (D) whenever $X^{**}$ is.
The necessary part of Theorem~\ref{Simonco:1} was proved by
Bueno and Svaiter in \cite[Lemma~3.1]{BuSv}. A similar result to Corollary~\ref{Simonco:2}\ref{STV:s1}
was also obtained by Bueno and Svaiter in \cite[Lemma~3.1]{BuSv} with
the additional assumption that $T$ be maximally monotone.
\end{remark}

 Theorem~\ref{PBABD:2} below
allows  us to  construct various maximally monotone operators
--- both linear and nonlinear --- that are not of type (D). The idea of constructing the operators in
the following
 fashion is based upon  \cite[Theorem~5.1]{BB} and was stimulated by \cite{BuSv}.
\begin{theorem}[Predual constructions]\emph{(See \cite[Theorem~3.7]{BBWY3}.)}
\label{PBABD:2}
Let $A: X^*\rightarrow X^{**}$ be linear and continuous.
 Assume that $\ran A \subseteq X$ and that there exists $e\in X^{**}\backslash X$ such that
\begin{align*}
\langle Ax^*,x^*\rangle=\langle e,x^*\rangle^2,\quad \forall x^*\in X^*.
\end{align*}
Let $ P$ and $S$ respectively  be the symmetric part and antisymmetric
 part of $A$.  Let $T:X\rightrightarrows X^*$  be defined by
\begin{align}\gra T&:=\big\{(-Sx^*,x^*)\mid x^*\in X^*, \langle e, x^*\rangle=0\big\}
=\big\{(-Ax^*,x^*)\mid x^*\in X^*, \langle e, x^*\rangle=0\big\}.\label{PBABA:a1}
\end{align}
Let $f:X\rightarrow\RX$ be a proper lower semicontinuous and convex function.
 Set $F:=f\oplus f^*$ on $ X\times X^*$.
Then the following hold.
\begin{enumerate}
\item\label{PBAB:em01}
$A$ is a maximally monotone operator on $X^*$ that is neither of type (D) nor unique.
\item\label{PBAB:emmaz1}
$Px^*=\langle x^*,e\rangle e,\ \forall x^*\in X^*.$

\item\label{PBAB:em1}
 $T$
is maximally monotone and skew on $X$.

\item\label{PBAB:emma1}
$\gra  T^*=\{(Sx^*+re,x^*)\mid x^*\in X^*,\ r\in\RR\}$.

\item\label{PBAB:emma2}
$-T$ is not maximally monotone.

\item\label{PBAB:emm1}
 $T$
is not of type (D).

\item\label{PBAB:em2}
$F_T=\iota_C$, where
$
C:=\{(-Ax^*,x^*)\mid x^*\in X^*\}$.

 \item\label{PBAB:emu2}
$T$ is not unique.

\item\label{PBAB:emr3}
$T$ is not of type (BR).

 \item\label{BCC:0a2}   If $\dom T\cap\inte\dom\partial f\neq\varnothing$,
then $T+\partial f$ is maximally monotone.
\item \label{BCC:02}$F$ and $F_T$ are BC--functions  on $X\times X^*$.
\item\label{BCC:2}Moreover,  \begin{align*}\bigcup_{\lambda>0} \lambda
\big(P_{X^*}(\dom F_T)-P_{X^*}(\dom F)\big)=X^*,\end{align*} while,
assuming that there exists $(v_0,v_0^*)\in X\times X^*$ such that
\begin{align}
f^*(v_0^*)+f^{**}(v_0-A^*v^*_0)<\langle
v_0,v^*_0\rangle\label{IeSp:3},
\end{align} \index{BC--function} then
$F_T\Box_1F$ is not a BC--function.

 \item\label{BCC:3} Assume that
$\left[\ran A-\bigcup_{\lambda>0} \lambda\dom f\right]$
is a closed subspace of $X$ and that
$$\varnothing\neq\dom f^{**}\small\circ A^*|_{X^*}\nsubseteqq \{e\}_{\bot}.$$
Then $T+\partial f$ is not of type (D).

\item\label{BCC:4}
Assume that $\dom f^{**}=X^{**}$.
Then $T+\partial f$ is a maximally monotone operator that is not of type (D).
\end{enumerate}
\end{theorem}

\begin{remark}\label{RProBR:1} Let $A$ be defined as in Theorem~\ref{PBABD:2}
By Proposition~\ref{ProBR1}, $A$ is not of type (BR) and then  Fact~\ref{MAS:BR1} implies that
$A$ is not of type (D). Moreover, $P_{X^*}\left[\dom F_A\right]\nsubseteq\overline{\ran A}$.
\end{remark}

The first application of this result is to $c_0$.

\begin{example}[$c_0$]\label{FPEX:1}{(See \cite[Example~4.1]{BBWY3}.)}
 Let $X: = c_0$, with norm $\|\cdot\|_{\infty}$ so that
  $X^* = \ell^1$ with norm $\|\cdot\|_{1}$, and  $X^{**}=\ell^{\infty}$  with its second dual norm
$\|\cdot\|_{*}$.
Fix
$\alpha:=(\alpha_n)_{n\in\NN}\in\ell^{\infty}$ with $\limsup
\alpha_n\neq0$, and let
$A_{\alpha}:\ell^1\rightarrow\ell^{\infty}$ be defined  by
\begin{align}\label{def:Aa}
(A_{\alpha}x^*)_n:=\alpha^2_nx^*_n+2\sum_{i>n}\alpha_n \alpha_ix^*_i,
\quad \forall x^*=(x^*_n)_{n\in\NN}\in\ell^1.\end{align}
\allowdisplaybreaks Now let $ P_{\alpha}$ and $S_{\alpha}$
respectively
  be the symmetric part and antisymmetric
 part of $A_{\alpha}$.  Let $T_{\alpha}:c_{0}\rightrightarrows X^*$  be defined by
\begin{align}\gra T_{\alpha}&
:=\big\{(-S_{\alpha} x^*,x^*)\mid x^*\in X^*,
 \langle \alpha, x^*\rangle=0\big\}
=\big\{(-A_{\alpha} x^*,x^*)\mid x^*\in X^*,
 \langle \alpha, x^*\rangle=0\big\}\nonumber\\
&=\big\{\big((-\sum_{i>n}
\alpha_n \alpha_ix^*_i+\sum_{i<n}\alpha_n \alpha_ix^*_i)_{n\in\NN}, x^*\big)
\mid x^*\in X^*, \langle \alpha, x^*\rangle=0\big\}.\label{PBABA:Ea1}
\end{align}
Then
\begin{enumerate}
\item\label{BCCE:A01} $\langle A_{\alpha}x^*,x^*\rangle=\langle \alpha , x^*\rangle^2,
\quad \forall x^*=(x^*_n)_{n\in\NN}\in\ell^1$
and \eqref{PBABA:Ea1} is well defined.

\item \label{BCCE:SA01} $A_{\alpha}$ is a maximally monotone
 operator on $\ell^1$ that is neither of type (D) nor unique.

\item\label{BCCE:A1} $T_{\alpha}$
is a maximally monotone  operator on $c_0$ that is not of type (D).
Hence  $c_0$  is not of type (D).
\item\label{BCCE:Ac1} $-T_{\alpha}$
is  not  maximally monotone.

\item \label{BCCE:Au2} $T_{\alpha}$ is neither unique nor of type (BR).

\item\label{BCCE:A2} $F_{T_{\alpha}}\Box_1
 (\|\cdot\|\oplus\iota_{B_{X^*}})$ is not a BC--function.\index{BC--function}
\item  \label{BCCE:A3}  $T_{\alpha}+
\partial \|\cdot\|$ is a maximally monotone operator on $c_0(\NN)$ that is not of type (D).
\item\label{BCCE:A4} If $\tfrac{1}{\sqrt{2}}<\|\alpha\|_*\leq 1$, then
$F_{T_{\alpha}}\Box_1 (\tfrac{1}{2}\|\cdot\|^2\oplus \tfrac{1}{2}\|\cdot\|^2_{1})$
 is not a BC--function.
\item \label{BCCE:A5}  For  $\lambda>0$,  $T_{\alpha}+\lambda J$
 is a maximally monotone operator on $c_0$ that is not of type (D).
\item\label{BCCE:A5a}  Let $\lambda>0$ and
 a linear isometry $L$
 mapping $c_0$ to  a subspace of $C[0,1]$ be given. Then
both  $(L^*)^{-1}(T_{\alpha}+\partial \|\cdot\|)L^{-1}$ and
 $(L^*)^{-1}(T_{\alpha}+\lambda J)L^{-1}$ are
  maximally monotone operators that are  not of type (D).
Hence $C[0,1]$ is not of type (D).
\item \label{BCCE:A06} Every Banach space that contains an isomorphic copy of $c_0$ is not of type (D).

\item \label{BCCE:A6}  Let $G:\ell^1\rightarrow\ell^{\infty} $
 be \emph{Gossez's operator } \cite{Gossez1} defined by
\begin{align*}
\big(G(x^*)\big)_n:=\sum_{i>n}x^*_i-
\sum_{i<n}x^*_i,\quad \forall(x^*_n)_{n\in\NN}\in\ell^1.
\end{align*}
Then $T_e: c_0\To\ell^1$ as defined by
\begin{align*}
\gra T_e:=\{(-G(x^*),x^*)\mid x^*\in\ell^1, \langle x^*, e\rangle=0\}
\end{align*}
is a maximally monotone operator that is not of type (D), where
$e:=(1,1,\ldots,1,\ldots)$.

\item \label{BCCE:A7} Moreover, $G$ is a  unique maximally monotone operator that is not of type (D),
but $G$ is isomorphically of type (BR). \qede
\end{enumerate}
\end{example}
\begin{remark}\label{RProBR:2}
Let $A_{\alpha}$ be defined as in Example~\ref{FPEX:1}.
By Remark~\ref{RProBR:1}, $A_{\alpha}$ is not of type (BR) and $P_{X^*}\left[\dom F_{A_{\alpha}}\right]\nsubseteq\overline{\ran A}$.
\end{remark}

\begin{remark}
The maximal monotonicity of the operator $T_{e}$ in Example~\ref{FPEX:1}\ref{BCCE:A6}
 was also  verified by Voisei and Z{\u{a}}linescu in
\cite[Example~19]{VZ} and  later a direct proof was given by Bueno and
Svaiter in  \cite[Lemma~2.1]{BuSv}. Herein we have given a new
 proof of the above results.

 Bueno and Svaiter had already showed  that
$T_{e}$ is not of type (D) in \cite{BuSv}. They  also showed
 that each Banach space that contains an isometric (isomorphic) copy of $c_0$
  is not of type (D) in \cite{BuSv}.
Example~\ref{FPEX:1}\ref{BCCE:A06} recaptures their result, while
Example~\ref{FPEX:1}\ref{BCCE:A2}\&\ref{BCCE:A4}
 provide a negative answer to Simons'
    \cite[Problem~22.12]{Si2}.
    \qede
\end{remark}

In our earlier work we were able to especially exploit some properties of the \emph{quasi-reflexive} James
space ($\J$ is of codimension one in $\J^{**}$):

\begin{definition}
The \emph{James space}, $\J$,  consists of all the sequences
$x=(x_n)_{n\in\NN}$ in $c_0$ with the finite norm
\begin{align*}
\|x\|:=\sup_{n_1<\cdots<n_k}\big((x_{n_1}-x_{n_2})^2+(x_{n_2}-x_{n_3})^2+
\cdots+(x_{n_{k-1}}-x_{n_k})^2\big)^{\tfrac{1}{2}}.
\end{align*}
\end{definition}
\begin{corollary}[Higher duals]\emph{(See \cite[Corollary~4.14]{BBWY3}.)} Suppose that both $X$ and $X^*$  admit maximally monotone operators not of type (D) then so does every higher dual space $X^{n}$. In particular, this applies to both $X=c_0$ and $X=\J$.
\end{corollary}

\section{Equivalence of three types of monotone operators}\label{s:Thrmon}

We now show that the three monotonicities  notions, of dense type (type (D)), negative infimum type and Fitzpatrick-Phelps
type  all
coincide.

\subsection{Type (NI) implies type (D)}
We reprise  two recent proofs---by  Marques Alves--Svaiter  and by Simons---showing that every maximally monotone operator of negative infimum type defined on a real Banach space is actually of dense type. We do this since the result is now central to current research and deserves being made as accessible as possible.

The key to establishing (NI) implies  (D) is in connecting the second conjugate of a convex function to the original function.
The next Fact is well known in various forms, but since we work with many dual spaces, care has to be taken in writing down the details.

\begin{fact}[Marques Alves and Svaiter]\emph{(See
\cite[Lemma~4.1]{MarSva}.)}
\label{NDMSV:2}
Let $F:X\times X^*\rightarrow\RX$ be proper lower semicontinuous  convex. Then
\begin{align}
F^{**}(x^{**}, x^*)=\liminf_{(y, y^*)\weakstarlynorm(x^{**}, x^*)}F(y, y^*),\quad \forall(x^{**}, x^*)\in
X^{**}\times X^*.\label{RPMSV:a1}
\end{align}

\end{fact}

\begin{proof}
Suppose that \eqref{RPMSV:a1} fails. Then there exists
$(x^{**}, x^*)\in
X^{**}\times X^*$ such that
\begin{align}F^{**}(x^{**}, x^*)\neq\liminf_{(y, y^*)\weakstarlynorm(x^{**}, x^*)}F(y, y^*).\label{RPMSV:a2}
\end{align}
Since $F^{**}=F$ on $X\times X^*$, then
\begin{align*}
F^{**}(x^{**}, x^*)\leq\liminf_{(y, y^*)\weakstarlynorm(x^{**}, x^*)}F^{**}(y, y^*)=\liminf_{(y, y^*)\weakstarlynorm(x^{**}, x^*)}F(y, y^*),
\end{align*}
and by\eqref{RPMSV:a2},
\begin{align}
F^{**}(x^{**}, x^*)<
\liminf_{(y, y^*)\weakstarlynorm(x^{**}, x^*)}F(y, y^*).\label{RPMSV:ab1}
\end{align}
Hence, $((x^{**}, x^*), F^{**}(x^{**}, x^*))\notin \overline{\eph F}^{w^{**}\times\|\cdot\|\times\|\cdot\|}$,
where  $\overline{\eph F}^{w^{**}\times\|\cdot\|\times\|\cdot\|}$ is
the weak$^{*}$-topology $\omega (X^{**}, X^*)\times \|\cdot\|\times \|\cdot\|$ in $X^{**}\times X^*\times\RR$.
By the Hahn-Banach Separation theorem, there exist $((z^{*}, z^{**}), \lambda)\in X^*\times X^{**}\times\RR$ and $r\in\RR$ such that
\begin{align}
\langle y, z^*\rangle+\langle y^*, z^{**}\rangle+\langle F(y, y^*)+t, -\lambda\rangle<r<
\langle x^{**}, z^*\rangle+\langle x^*, z^{**}\rangle+\langle F^{**}(x^{**}, x^*), -\lambda\rangle,
\label{RPMSV:a3}
\end{align}
for every $(y,y^*)\in\dom F$ and $t\geq0$.
Hence $\lambda\geq0$. Next we show
$\lambda\neq0.$
Suppose to the contrary that $\lambda=0$. Then \eqref{RPMSV:a3} implies that
\begin{align}
\sup_{(y,y^*)\in\dom F}\langle y, z^*\rangle+\langle y^*, z^{**}\rangle\leq r<
\langle x^{**}, z^*\rangle+\langle x^*, z^{**}\rangle.
\label{RPMSV:a4}
\end{align}
As in \cite[Section~6]{Moreau}, $\dom F^{**}$  is a subset of the closure of $\dom F$ in the topology $\omega (X^{**}, X^*)\times \omega (X^{***}, X^{**}),$
and so \begin{align*}
\sup_{(y,y^*)\in\dom F}\langle y, z^*\rangle+\langle y^*, z^{**}\rangle\geq
\langle x^{**}, z^*\rangle+\langle x^*, z^{**}\rangle
\end{align*}since $(x^{**}, x^*)\in\dom F^{**}$ by \eqref{RPMSV:ab1}. This contradicts
\eqref{RPMSV:a4}.
Hence $\lambda\neq0$ and so $\lambda>0$.

Taking $t=0$ and the supremum on $(y,y^*)$ in \eqref{RPMSV:a3}, we deduce
\begin{align*}
\lambda F^*(\tfrac{z^*}{\lambda},\tfrac{z^{**}}{\lambda})= (\lambda F)^* (z^*, z^{**})
<
\langle x^{**}, z^*\rangle+\langle x^*, z^{**}\rangle-\lambda F^{**}(x^{**}, x^*),
\end{align*}
and so we have
\begin{align}
\lambda F^*(\tfrac{z^*}{\lambda},\tfrac{z^{**}}{\lambda})+\lambda F^{**}(x^{**}, x^*)
<
\langle x^{**}, z^*\rangle+\langle x^*, z^{**}\rangle.\label{RPMSV:a5}
\end{align}
But by Fenchel-Young inequality \cite{BorVan},
\begin{align*}
\lambda F^*(\tfrac{z^*}{\lambda},\tfrac{z^{**}}{\lambda})+\lambda F^{**}(x^{**}, x^*)
\geq \lambda\big(\langle \tfrac{z^*}{\lambda}, x^{**}\rangle+\langle \tfrac{z^{**}}{\lambda}, x^{*}\rangle\big)
=\langle x^{**}, z^*\rangle+\langle x^*, z^{**}\rangle,
\end{align*}
which contradicts \eqref{RPMSV:a5}.
Hence \eqref{RPMSV:a1} holds.\end{proof}

Let $CLB(X)$ denote the set of all convex functions from $X$ to $\RR$ that are Lipschitz on the bounded subsets of $X$, and
$ \mathrm{T}_{CLB}(X^{**})$ on $X^{**}$ be the topology on $X^{**}$ such that $h^{**}$ is continuous everywhere on
$X^{**}$ for every $h\in CLB(X)$. (For more information about $CLB(X)$ and $ \mathrm{T}_{CLB}(X^{**})$ see  \cite{Si2}.)

This prepares us for the main tool in the proof of Theorem~\ref{NDMSV:3}.
 We describe two proofs, the first is due to Simons, and  the second is due to Marques Alves and Svaiter. The goal is to restrict ourselves to bounded nets.

\begin{theorem}[Marques Alves and Svaiter]\emph{(See
\cite[Theorem~4.2]{MarSva}.)}\label{NDMSV:3}
Let $F:X\times X^*\rightarrow\RX$ be proper, (norm) lower semicontinuous and convex. Consider $(x^{**}, x^*)\in
X^{**}\times X^*$. Then there exists a bounded net $(x_{\alpha}, x^*_{\alpha})_{\alpha\in I}$ in $X\times
X^*$  that converges to $(x^{**}, x^*)$ in the weak$^{*}$-topology $\omega (X^{**}, X^*)\times \|\cdot\|$
such that
\begin{align*}
F^{**}(x^{**}, x^*)=\lim F(x_{\alpha}, x^*_{\alpha}).
\end{align*}
\end{theorem}

\begin{proof}  The first proof, due to Simons, is very concise but quite abstract.

\noindent \emph{Proof one}:
By \cite[Lemma~45.9(a)]{Si2}, there exists a net $(x_{\alpha}, x^*_{\alpha})_{\alpha\in I}$ in $X\times
X^*$  that converges to $(x^{**}, x^*)$ in the $ \mathrm{T}_{CLB}(X^{**}\times X^{***})$ such that
\begin{align*}
F^{**}(x^{**}, x^*)=\lim F(x_{\alpha}, x^*_{\alpha}).
\end{align*}
By \cite[Lemma~49.1(b)]{Si2} (apply $X^*$ to $H$) (or \cite[Lemma 7.3(b)]{Si5}), $(x_{\alpha}, x^*_{\alpha})_{\alpha\in I}$  converges to $(x^{**}, x^*)$ in the $ \mathrm{T}_{CLB}(X^{**})\times \|\cdot\|$.
Then by \cite[Lemma~7.1(a)]{Si5},
$(x_{\alpha}, x^*_{\alpha})_{\alpha\in I}$ is eventually bounded and
 converges to $x^{**}$ in the weak$^{*}$-topology $\omega (X^{**}, X^*)\times\|\cdot\|$.

The second proof, due to Marques Alvez and Svaiter, is equally concise but more direct.

\noindent \emph{Proof two}:  We consider two cases.\\
\noindent \emph{Case 1}: $(x^{**}, x^*)\in\dom F^{**}$.
Fix  $M>0$  such that $\|(x^{**}, x^*)\|\leq M$ and $\dom F\cap\inte MB_{X\times X^*}\neq\varnothing$, where $B_{X\times X^*}$ is a closed unit ball of $X\times X^*$.
Then as in the proof of  \cite[Eqn. (2.8) of Prop.~1]{Rock702}, \cite[Theorem 4.4.18]{BorVan} or \cite[Lemma~2.3]{MarSva},
\begin{align*}(F+\iota_{MB_{X\times X^*}})^{**}(x^{**}, x^*)&=\left[F^*\Box \iota^*_{MB_{X\times X^*}}\right]^*(x^{**}, x^*)
=(F^{**}+\iota^{**}_{MB_{X\times X^*}})(x^{**}, x^*)\\
&=F^{**}(x^{**}, x^*).
\end{align*}
Finally, we may directly apply  Fact~\ref{NDMSV:2} with $F+\iota_{MB_{X\times X^*}}$ replacing $F$.

\noindent \emph{Case 2}: $(x^{**}, x^*)\notin\dom F^{**}$.  Then $F^{**}(x^{**}, x^*)=+\infty$. By  Goldstine's theorem, there exists a bounded net
$(x_{\alpha})_{\alpha\in I}$ in $X$ such that $(x_{\alpha})_{\alpha\in I}$ converges to $x^{**}$ in the weak$^{*}$-topology $\omega (X^{**}, X^*)$. Take $x^*_{\alpha}=x^*,\forall \alpha\in I$. By  lower semicontinuity of $F^{**}$ and since
$F^{**}=F$ on $X\times X^*$, we have $\lim F(x_{\alpha}, x^*_{\alpha})=+\infty$, and hence we have
$F^{**}(x^{**}, x^*)=\lim F(x_{\alpha}, x^*_{\alpha})$ as asserted.
\end{proof}

\begin{remark}
In Theorem~\ref{NDMSV:3}, the result cannot hold if we select points from
$X^{**}\times X^{***}$.  For example, suppose that $X$ is nonreflexive.  Thus $B_{X^*}\varsubsetneq
B_{X^{***}}$.  Define $F:=\iota_{\{0\}}\oplus \iota_{B_{X^*}}$ on $X\times X^*$.
Thus $F^{**}=\iota_{\{0\}}\oplus \iota_{B_{X^{***}}}$ on $X^{**}\times X^{***}$.
Take $x_0^{***}\in B_{X^{***}}\backslash B_{X^*}$. Thus $F^{**}(0,x_0^{***})=0$.
Suppose that there exist a bounded net $(x_{\alpha}, x^*_{\alpha})_{\alpha\in I}$ in $X\times
X^*$  that converges to $(0, x_0^{***})$ in the weak$^{*}$-topology $\omega (X^{**}, X^*)\times \|\cdot\|$
such that
\begin{align*}
0=F^{**}(0, x_0^{***})=\lim F(x_{\alpha}, x^*_{\alpha})=\lim\big(\iota_{\{0\}}\oplus \iota_{B_{X^*}}\big)(x_{\alpha}, x^*_{\alpha}).
\end{align*}
Thus there exists $\alpha_0\in I$ such that $(x_{\alpha}, x^*_{\alpha})\in\{0\}\times B_{X^*}$ for every $
\alpha\succeq_I\alpha_0$.  Thus $x_0^{***}\in B_{X^*}$, which contradicts that $x_0^{***}\notin B_{X^*}$.
\end{remark}

Similar arguments to  those of the first proof of
Theorem~\ref{NDMSV:3} lead to:

\begin{proposition}\label{FPV:PR4}
Let $F:X\times X^*\rightarrow\RX$ be proper, lower semicontinuous and convex. Let $(x, x^{***})\in
X\times X^{***}$. Then there exists a bounded net $(x_{\alpha}, x^*_{\alpha})_{\alpha\in I}$ in $X\times
X^*$  that converges to $(x, x^{***})$ in the norm $\times$ weak$^{*}$-topology $\omega (X^{***}, X^{**})$
such that
\begin{align*}
F^{**}(x, x^{***})=\lim F(x_{\alpha}, x^*_{\alpha}).
\end{align*}
\end{proposition}

We now apply Theorem~\ref{NDMSV:3} to \emph{representative} functions attached to $A$;  more precisely, to the \emph{Fitzpatrick function}:
\begin{align}\label{defFA}F_A(x, x^*):= \sup_{a^* \in Aa}\Big( \langle x,a^*\rangle+\langle a,x^*\rangle-\langle a,a^*\rangle\Big).
\end{align}
Let $Y$ be another real Banach space and  $F:X\times Y\rightarrow\RX$. We define $F^{\intercal}:Y\times X\rightarrow\RX$ by
\begin{align*}F^{\intercal}(y, x):= F(x,y),\quad \forall (x,y)\in X\times Y.
\end{align*}

Three more building blocks follow:

\begin{fact}[Fitzpatrick]
\emph{(See (\cite[Propositions~3.2\&4.1, Theorem~3.4 
]{Fitz88}.)}
\label{f:Fitz}
Let $Y$ be a real Banach space and $A\colon Y\To Y^*$ be monotone with $\dom A\neq\varnothing$.
Then $F_A$ is proper lower semicontinuous, convex, and $(F_A)^{*\intercal}=F_A=\langle\cdot,\cdot\rangle$ on $\gra A$.
\end{fact}

\begin{fact}[Fitzpatrick]
\emph{(See {\cite[Theorem~3.10]{Fitz88}}.)}
\label{f:Fitzp2}
Let $A\colon X\To X^*$ be  maximally monotone, and let $F:X\times X^*\rightarrow\RX$ be convex.
Assume that $F\geq \langle\cdot,\cdot\rangle$ on $X\times X^*$ and $F=\langle\cdot,\cdot\rangle$ on $\gra A$.
Then $F_A\leq F$.
\end{fact}

\begin{fact} [Simons]\emph{(See \cite[Theorem~19(c)]{SiNI}.)}\label{retyD:1}
Let $A:X\rightrightarrows X^*$ be  maximally monotone of type (NI). Then
the operator $B: X^*\rightrightarrows X^{**}$ defined by
\begin{align}
\gra B:=\big\{(x^*, x^{**})\in X^*\times X^{**}\mid \langle x^*-a^*, x^{**}-a\rangle\geq0,\,\forall (a,a^*)\in\gra A\big\}\label{EretyD:1}
\end{align}
is maximally monotone.
\end{fact}

We can now establish an important link between $A$ and $B$:

\begin{proposition} \label{retyD:2}
Let $A:X\rightrightarrows X^*$ be  maximally monotone of type (NI), and let
the operator $B: X^*\rightrightarrows X^{**}$ be  defined as in  \eqref{EretyD:1}.
Then $(F_B)^{\intercal}\leq(F_A)^{**}\leq (F_B)^{*}$ on $X^{**}\times X^{*}$.
In consequence, $(F_A)^{**}=\langle\cdot,\cdot\rangle$ on $\gra B$.
\end{proposition}
\begin{proof}
Define $F:X\times X^*\rightarrow\RX$ by
\begin{align*}
(y,y^*)\mapsto\langle y,y^*\rangle+\iota_{\gra A}(y,y^*).
\end{align*}
We have
\begin{align}
F^*(x^*, x^{**})&=\sup_{(a,a^*)\in\gra A}\big\{ \langle x^*,a\rangle+\langle x^{**}, a^*\rangle-\langle a,a^*\rangle     \big\}\nonumber\\
&=\langle x^*, x^{**}\rangle-\inf_{(a,a^*)\in\gra A}\langle x^{**}-a, x^*-a^*\rangle,\quad \forall (x^*, x^{**})\in X^*\times X^{**}.\label{retyD:2Eaa1}
\end{align}
Since $A$ is of type (NI), $\inf_{(a,a^*)\in\gra A}\langle x^{**}-a, x^*-a^*\rangle\leq0$ and  so \eqref{retyD:2Eaa1} shows
\begin{align}
F^*(x^*, x^{**})\geq \langle  x^*, x^{**}\rangle,\quad\forall(x^*, x^{**})
 \in X^*\times X^{**}.\label{retyD:2Eaa2}
 \end{align}
 Now when $(x^*, x^{**})\in\gra B$, then $\inf_{(a,a^*)\in\gra A}\langle x^{**}-a, x^*-a^*\rangle\geq0$. Hence,
 $F^*(x^*, x^{**})\leq \langle  x^*, x^{**}\rangle$ by \eqref{retyD:2Eaa1} again. Combining this with \eqref{retyD:2Eaa2}, leads to
 $F^*(x^*, x^{**})= \langle  x^*, x^{**}\rangle$.
 Thus, $F^*\geq \langle\cdot,\cdot\rangle$ on $X^*\times X^{**}$ and $F^*=\langle\cdot,\cdot\rangle$ on $\gra B$.
 Thence,
 Fact~\ref{f:Fitzp2} and Fact~\ref{retyD:1} imply that \begin{align}
 F^*\geq F_B.\label{retyD:2Ea1}
 \end{align}

As $F_A=F$ on $\dom F=\gra A$, by Fact~\ref{f:Fitzp2} we have $F_A\leq F$ and so $(F_A)^*\geq F^*$. Hence, by \eqref{retyD:2Ea1},
\begin{align}(F_A)^*\geq F_B.\label{retyD:2Eb1}
\end{align}
By Fact~\ref{f:Fitz}, $(F_A)^*=\langle\cdot,\cdot\rangle$ on $\gra A^{-1}$ and so we have
$H:=\langle \cdot,\cdot\rangle+\iota_{\gra A^{-1}}\geq (F_A)^*$ on $X^{*}\times X^{**}$.
Thus by \eqref{retyD:2Eb1},
\begin{align}H \geq (F_A)^*\geq F_B\quad\text{on}\,X^{*}\times X^{**}.
\label{reD:2c1}
 \end{align}
 Fix $(x^*, x^{**})
 \in X^*\times X^{**}$.
 On conjugating we obtain $F^*(x^{*}, x^{**})=H^*(x^{**}, x^*)\leq(F_A)^{**}(x^{**}, x^*)\leq (F_B)^*(x^{**}, x^{*})$ and so by
\eqref{retyD:2Ea1},
\begin{align}F_B(x^{*}, x^{**})\leq
(F_A)^{**}(x^{**}, x^*)\leq (F_B)^*(x^{**}, x^{*})
\end{align}

To conclude the proof, if $(x^{*}, x^{**})\in\gra B$, then
 by Fact~\ref{f:Fitz} (applied in $Y=X^*$ on appealing to Fact~\ref{retyD:1})
  $F_B(x^{*}, x^{**})=\langle x^{*}, x^{**}\rangle=(F_B)^*(x^{**}, x^{*})$.
  Then $(F_A)^{**}(x^{**}, x^*)=\langle x^{**}, x^*\rangle$
  and
 as asserted
 $(F_A)^{**}=\langle\cdot,\cdot\rangle$ on $\gra B$.
\end{proof}

\begin{fact}[Simons / Voisei and Z{\u{a}}linescu]\emph{(See
\cite[Theorem 2.6]{VZ:1} and \cite[Theorem~4.4]{MarSva}, \cite[Theorem~6.12(a) and Remark~5.13]{Si5}
 or \cite[Theorem 9.7.2]{BorVan}.)} \label{StFiz} Let $A:X \To X^*$
be a maximally  monotone operator that is of type (NI).
 Then
 \begin{align*}
 F_A(x,x^*)-\langle x,x^*\rangle &\ge
 \tfrac{1}{2}\inf_{(a,a^*)\in\gra A}\{\|x-a\|^2+ \|x^*-a^*\|^2\} ,
 \end{align*}
 and hence $ F_A(x,x^*)-\langle x,x^*\rangle\ge
 \tfrac{1}{4}\inf_{(a,a^*)\in\gra A}\{\|x-a\|^2+\|x^*-a^*\|^2\}$.
\end{fact}
The first inequality was first established by Simons, and the second inequality was first established by
Voisei and Z{\u{a}}linescu.

The proof of Theorem~\ref{retyD:3} now follows the lines of that of \cite[Theorem~6.15~(c)$\Rightarrow$(a)]{Si5}.

\begin{theorem}[Marques Alves and Svaiter] \emph{(See
\cite[Theorem~4.4]{MarSva} or \cite[Theorem~9.5]{Si5}.)}\label{retyD:3}
Let $A:X\rightrightarrows X^*$ be  maximally monotone. Assume that $A$ is of type (NI).
Then
$A$ is of type (D).
\end{theorem}
\begin{proof}
Let $(x^{**}, x^*)\in X^{**}\times X^*$ be monotonically
related to $\gra A$, and
let $B$ be defined by \eqref{EretyD:1} of Fact~\ref{retyD:1}. Then $(x^*, x^{**})\in \gra  B$.
Thus by Proposition~\ref{retyD:2}, $(F_A)^{**}(x^{**}, x^*)=\langle x^{**}, x^*\rangle$.
By Theorem~\ref{NDMSV:3},
 there exists a bounded net $(x_{\alpha}, x^*_{\alpha})_{\alpha\in I}$ in $X\times
X^*$  that converges to $(x^{**}, x^*)$ in the weak$^{*}$-topology $\omega (X^{**}, X^*)\times \|\cdot\|$
such that
\begin{align*}
\langle x^{**}, x^*\rangle=(F_A)^{**}(x^{**}, x^*)=\lim F_A (x_{\alpha}, x^*_{\alpha}).
\end{align*}
By Fact~\ref{f:Fitz},
\begin{align*}
\langle x^{**}, x^*\rangle=\lim F_A (x_{\alpha}, x^*_{\alpha})\geq \lim\langle x_{\alpha}, x^*_{\alpha}\rangle=\langle x^{**}, x^*\rangle.
\end{align*}
Then we have
$0\leq\varepsilon_{\alpha}:=F_A (x_{\alpha}, x^*_{\alpha})-\langle x_{\alpha}, x^*_{\alpha}\rangle\rightarrow0$.
By Fact~\ref{StFiz},  there exists a net $(a_{\alpha}, a^*_{\alpha})_{\alpha\in I}$ in $\gra A$ with
\begin{align*}
 \|a_{\alpha}-x_{\alpha}\|\leq 2\sqrt{\varepsilon_{\alpha}}\quad\text{and}\quad
\|a^*_{\alpha}-x^*_{\alpha}\|\leq 2\sqrt{\varepsilon_{\alpha}}.
\end{align*}
So,  $(a_{\alpha}, a^*_{\alpha})_{\alpha\in I}$  is eventually bounded and converges to $(x^{**}, x^*)$ in the topology $\omega (X^{**}, X^*)\times \|\cdot\|$ as required.
\end{proof}

\subsection{Type (FP) implies type (NI)}

This section  relies on work in \cite{BBWY2} by Bauschke, Borwein, Wang and Yao.

\begin{fact}[Simons]
\emph{(See \cite[Theorem~17]{Si4} or \cite[Theorem~37.1]{Si2}.)}
\label{FTSim:1}
Let $A:X\To X^*$ be maximally monotone
and of type (D). Then $A$ is of type (FP).
\end{fact}

\begin{theorem}\label{FP:1}\emph{(See \cite[Theorem~3.1]{BBWY2}.)}
Let $A:X\To X^*$ be maximally monotone such that
$A$ is of type (FP).
Then $A$ is of type (NI).
\end{theorem}

We now obtain the promised corollary:

\begin{corollary}\emph{(See \cite[Corollary~3.2]{BBWY2}.)}
\label{cor:main} Let $A\colon X\To X^*$ be maximally monotone. Then
the following are equivalent.
\begin{enumerate}
\item\label{MaFT:1} $A$ is of type (D).
\item \label{MaFT:2} $A$ is of type (NI).
\item \label{MaFT:3} $A$ is of type (FP).
\end{enumerate}
\end{corollary}
\begin{proof}
``\ref{MaFT:1}$\Rightarrow$\ref{MaFT:3}":  Fact~\ref{FTSim:1}. ``\ref{MaFT:3}$\Rightarrow$\ref{MaFT:2}":
 Theorem~\ref{FP:1} . ``\ref{MaFT:2}$\Rightarrow$\ref{MaFT:1}":
Theorem~\ref{retyD:3}.
\end{proof}

We turn to the  construction of linear and nonlinear maximally monotone operators that are of type (BR)  but not of type (D).

\begin{proposition}\label{CSBE:1}
   Let  $A: X\rightrightarrows X^*$ be  maximally monotone, and
 let $Y$ be another real Banach space and $B: Y\rightrightarrows Y^*$ be  maximally monotone of type (BR).
Assume that $A$ is of type (BR) but not of type (D). Define the norm on $X\times Y$ by $\|(x,y)\|:=\|x\|+\|y\|$.
Let $T: X \times Y \rightrightarrows X^*\times Y^*$ be defined
   by $T(x,y):=\big(Ax,By\big)$. Then $T$ is a maximally monotone operator that is of type (BR) but not of type (D).
   In consequence, if $A$ and $B$ are actually isomorphically (BR) then so is $T$.
\end{proposition}
\begin{proof}
We first show that $T$ is maximally monotone.
Clearly, $T$ is monotone. Let $\big((x,y), (x^*,y^*)\big)\in (X\times Y)\times (X^*\times Y^*)$ be
monotonically related to $\gra T$. Then we have
\begin{align*}
\langle x-a, x^*-a^*\rangle+
\langle y-b, y^*-b^*\rangle\geq0,\quad \forall (a,a^*)\in\gra A,\
(b,b^*)\in\gra A.
\end{align*}
Thus
\begin{align}
\inf_{(a,a^*)\in\gra A}\langle x-a, x^*-a^*\rangle+
\inf_{(b,b^*)\in\gra B}\langle y-b, y^*-b^*\rangle\geq0.\label{MBRP:1}
\end{align}
Let $r:=\inf_{(a,a^*)\in\gra A}\langle x-a, x^*-a^*\rangle$.
We consider two cases.

\emph{Case 1}: $r\geq0$.
By maximal monotonicity of $A$, we have $(x,x^*)\in\gra A$. Then $r=0$.
Thus by \eqref{MBRP:1} and maximal monotonicity of $B$,
$(y,y^*)\in\gra B$ and hence $\big((x,y), (x^*,y^*)\big)\in\gra T$.

\emph{Case 2}: $r<0$.
Thus by \eqref{MBRP:1}
\begin{align}
\inf_{(b,b^*)\in\gra B}\langle y-b, y^*-b^*\rangle\geq-r>0.\label{MBRP:2}
\end{align}
Since  $B$ is maximally monotone, $(y,y^*)\in\gra B$ and hence
\begin{align*}
\inf_{(b,b^*)\in\gra B}\langle y-b, y^*-b^*\rangle=0,
\end{align*}
which contradicts \eqref{MBRP:2}.

Combining  cases, we see that $T$ is maximally monotone.

Next we show $T$ is of type (BR)
Let  $\big((u,v), (u^*,v^*)\big)\in (X\times Y)\times (X^*\times Y^*)$, $\alpha,\beta>0$ such that
\begin{align}
\inf_{(a,a^*)\in\gra A}\langle u-a, u^*-a^*\rangle+
\inf_{(b,b^*)\in\gra B}\langle v-b, v^*-b^*\rangle
&=\inf_{(z,z^*)\in\gra T} \langle (u,v)-z,(u^*,v^*)-z^*\rangle\nonumber\\
&>-\alpha\beta. \label{MBRP:4}
\end{align}
Let $-\rho=\inf_{(b,b^*)\in\gra B}\langle v-b, v^*-b^*\rangle$. Then $\rho\geq0$.
We consider two cases.

\emph{Case 1}: $\rho=0$.
Then $(v,v^*)\in\gra B$.
By \eqref{MBRP:4} and since $A$ is of type (BR), there exists $(a_1,a^*_1)\in\gra A$ such  that
\begin{align*}
\|u-a_1\|<\alpha,\quad \|u^*-a^*_1\|<\beta.
\end{align*}
Thus
\begin{align*}
\|(u,v)-(a_1,v)\|=\|u-a_1\|<\alpha,\quad \|(u^*,v^*)-(a^*_1,v^*)\|=\|u^*-a^*_1\|<\beta.
\end{align*}

\emph{Case 2}: $\rho>0$.
By \eqref{MBRP:4},
\begin{align}
\inf_{(a,a^*)\in\gra A}\langle u-a, u^*-a^*\rangle
&>-\alpha\beta+\rho=-(\alpha-\frac{\rho}{\beta})\beta. \label{MBRP:5}
\end{align}
Then by the maximal monotonicity of $A$,
\begin{align*}
\inf_{(a,a^*)\in\gra A}\langle u-a, u^*-a^*\rangle\leq0.
\end{align*}
Then by \eqref{MBRP:5},
\begin{align}
\alpha-\frac{\rho}{\beta}>0. \label{MBRP:6}
\end{align}
Then by \eqref{MBRP:5} and by that $A$ is of type (BR),
there exists $(a_2,a^*_2)\in\gra A$ such that
\begin{align}
\|u-a_2\|<\alpha-\frac{\rho}{\beta},\quad \|u^*-a^*_2\|<\beta.
\label{MBRP:7}
\end{align}

Since $B$ is of type (BR), there exists $(b_2,b^*_2)\in\gra B$ such that
\begin{align}
\|v-b_2\|<\frac{\rho}{\beta},\quad \|v^*-b^*_2\|<\beta.\label{MBRP:8}
\end{align}

Taking $(z_1,z_1^*)=\big((a_2,b_2),(a^*_2,b^*_2)\big)$ and combing
\eqref{MBRP:7} and \eqref{MBRP:8}, we have $(z_1,z_1^*)\in\gra T$
\begin{align*}
\|(u,v)-z_1\|<\alpha,\quad \|(u^*,v^*)-z_1^*\|<\beta.
\end{align*}
Hence $T$ is of type (BR).

Lastly, we show $T$ is not of type (D).

Since $A$ is not of type (D), by Corollary~\ref{cor:main}, $A$ is not of type (NI) and then there exists $(x_0^{**},x_0^*)\in X^{**}\times X^*$ such that
\begin{equation}
\inf_{(a,a^*)\in\gra A}\langle x_0^{**}-a,x_0^*-a^*\rangle>0.\label{MBRP:10}
\end{equation}
Take $(y_0,y_0^*)\in\gra B$ and $(z^{**}_0,z_0)=(x_0^{**}+y_0,x_0^*+y^*_0)$. Then by \eqref{MBRP:10},
\begin{align*}
\inf_{(z,z^*)\in\gra T}\langle z_0^{**}-z,z_0^*-z^*\rangle=\inf_{(a,a^*)\in\gra A}\langle x_0^{**}-a,x_0^*-a^*\rangle>0.
\end{align*}
Hence $T$ is not of type (NI) and hence $T$ is not of type (D) by Corollary~\ref{cor:main}.

When $A$ and $B$ are actually isomorphically (BR), following the above proof, we see that $T$ is isomorphically (BR).
\end{proof}

\begin{corollary}\label{BRnotD}
 Let $B: \ell^2\rightrightarrows\ell^{2} $ be an arbitrary maximally monotone operator and define
 $T: \ell^1 \times\ell^2 \rightrightarrows \ell^1 \times\ell^{2}$
   by $T(x,y):=\big(G(x), B(x)\big)$, where $G$ is the Gossez operator.
   Then $T$ is a maximally monotone operator that is of type (BR) isomorphically but not of type (D).
\end{corollary}
\begin{proof}
Since $\ell^2$ is reflexive, $B$ is of type (BR). Then
apply Example~\ref{FPEX:1}\ref{BCCE:A7} and Proposition~\ref{CSBE:1} directly.
\end{proof}

\begin{remark}
In the case that $A$ in Proposition~\ref{CSBE:1} is nonaffine we obtain nonaffine operators of type (BR) which do not have unique extensions to the bidual, since, unless the operator is affine, uniqueness implies type (D) \cite{MarSva2}.
\end{remark}

In the next section, we will explore  properties of type (DV)
operators as defined below.  This a useful dual notion to type (D).

\subsection{Properties of type (DV) operators}\label{s:DFPV}

Let $A:X\rightrightarrows X^*$ be (maximally) monotone. We say $A$
is \emph{of type Fitzpatrick-Phelps-Veronas (FPV)} if  for every
open convex set $U\subseteq X$ such that $U\cap \dom
A\neq\varnothing$, the implication
\begin{align*}
x\in U\,\text{and}\,(x,x^*)\,\text{is monotonically related to $\gra A\cap (U\times X^*)$}\quad
\Rightarrow (x,x^*)\in\gra A
\end{align*}
holds.


We also introduce a dual definition of type (DV),
corresponding to the definition of type (D).

\begin{definition}
 Let $A:X\To X^*$ be maximally monotone.
We say  $A$ is
\emph{of type (DV)}  if for every
$(x,x^{***})\in X\times X^{***}$ with
\begin{align*}
\inf_{(a,a^*)\in\gra A}\langle a-x, a^*-x^{***}\rangle\geq 0,
\end{align*}
there exists a  bounded net
$(a_{\alpha}, a^*_{\alpha})_{\alpha\in\Gamma}$ in $\gra A$
such that
$(a_{\alpha}, a^*_{\alpha})_{\alpha\in\Gamma}$
converges to
$(x,x^{***})$ in the norm $\times$ the  weak$^{*}$-topology $\omega (X^{***}, X^{**})$.
\end{definition}

%
%
%
%
%

\begin{proposition}\label{FPDV:PR6}
Let $A:X\rightrightarrows X^*$ be a maximally monotone operator.  Let $B:X\rightrightarrows X^{***}$ be defined by $\gra B:=\big\{(a,a^{***})\mid (a, a^{***}|_X)\in\gra A\big\}$. Then $B$ is a unique maximally monotone extension of $\gra A$ in $X\times X^{***}$.
\end{proposition}

\begin{proof}
Clearly, $B$ is monotone with respect to $X\times X^{***}$. Now we show that $B$ is maximally monotone with respect to  $X\times X^{***}$.

Let $(x,x^{***})\in X\times X^{***}$ be monotonically related to $\gra A$.  Let $x^*:=x^{***}|_X$. Then
$x^*\in X^*$ such that
\begin{align*}
\langle a-x, a^*-x^{*}\rangle
=
\langle a-x, a^*-x^{***}\rangle \geq0,\quad\forall (a,a^*)\in\gra A.
\end{align*}
Since $A$ is maximally monotone, $(x, x^*)\in\gra A$
and then $(x, x^{***}|_{X})\in\gra A$.
Hence $(x,x^{***})\in\gra B$.
Thus, $\gra B$ contains all the elements in $X\times X^{***}$ that are monotonically related to $\gra A$. Since $\gra B$ is monotone,
 $B$ is a unique maximally monotone extension of $A$ in $X\times X^{***}$.
 \end{proof}

The next result will confirm  that type (DV) and  type (D) are
distinct. Neither every subdifferential operator nor  every linear continuous monotone operator is of type (DV).
In consequence, type (D) operators are  not always type (DV).
\begin{proposition}\label{FPV:PR5}
Let $A:X\rightarrow X^*$ be a continuous  maximally monotone operator.  Then
$A$ is  of type (DV) if and only if $X$ is reflexive.
\end{proposition}
\begin{proof}
``$\Leftarrow$": Clear.

``$\Rightarrow$": Suppose to the contrary that $X$ is not reflexive.  Since $X\varsubsetneq X^{**}$ and $X$ is a closed subspace of $X^{**}$,  by the Hahn-Banach theorem, there exists
$x^{***}_0 \in X^{***}\backslash\{0\}$ such that $\langle x^{***}_0, X\rangle=\{0\}$.
Then we have
\begin{align*}\big\langle a-0, Aa-(A0+x^{***}_0 )\big\rangle=
\big\langle a-0, Aa-A0\big\rangle-\big\langle a, x^{***}_0 \big\rangle=\big\langle a-0, Aa-A0\big\rangle
\geq0,\quad \forall a\in X.
\end{align*}
Then $(0, A0+ x^{***}_0)$ is monotonically related to $\gra A$.
Since $A$ is of type (DV), there exists
 a  bounded net
$(a_{\alpha}, Aa_{\alpha})_{\alpha\in\Gamma}$ in $\gra A$
such that
$(a_{\alpha}, Aa_{\alpha})_{\alpha\in\Gamma}$
converges to
$(0,A0+x_0^{***})$ in the norm $\times$ the  weak$^{*}$-topology $\omega (X^{***}, X^{**})$.
 Since $A$ is continuous and $a_{\alpha}\longrightarrow 0$, we have
 $Aa_{\alpha}\longrightarrow A0$ in $X^*$ and hence  $Aa_{\alpha}\longrightarrow A0$ in $X^{***}$ .
 Since $Aa_{\alpha}$ converges $A0+x_0^{***}$ in the  weak$^{*}$-topology $\omega (X^{***}, X^{**})$, we have $A0=A0+ x_0^{***}$ and hence $x_0^{****}=0$, which contradicts
 that $x^{***}_0 \in X^{***}\backslash\{0\}$.
 Hence $X$ is reflexive.
 \end{proof}

\begin{remark} Let $P_{\alpha}$ be defined in
Example~\ref{FPEX:1}, then  $P_{\alpha}$ is a subdifferential operator defined on $\ell^1$ and also a bounded continuous linear operator.  Then it is of type (D)
but it is not of type (DV) by Proposition~\ref{FPV:PR5}. Hence type (D) cannot imply
type (DV).
\end{remark}

\begin{remark} It is unknown whether every maximally monotone operator is of type (FPV).   Perhaps property (DV) may help shed light on the matter.
We have also been unable to determine if (DV) implies (FPV).
\end{remark}

We  might say  a Banach space $X$ is \emph{of type (DV)}
if every maximally monotone operator on $X$ is necessarily of  type (DV).

\begin{theorem}
The Banach space $X$ is of type (DV) if and only if it is reflexive.
\end{theorem}
\begin{proof}
``$\Leftarrow$": Clear.

``$\Rightarrow$":  Let $A:X\rightarrow X$ defined by $\gra A:=X\times\{0\}$. Then $A$ is maximally monotone continuous linear operator. Since $A$ is of type (DV), Proposition~\ref{FPV:PR5} implies that $X$ is reflexive.
\end{proof}


Finally, we give an example of a type (DV) operator in an arbitrary Banach space.

\begin{example}
 Let $X$ be a Banach space and let $A:X\rightrightarrows X^*$ be defined by $\gra A:=\{0\}\times X^*$.
 Let $B$ be defined by $\gra B:=\{0\}\times X^{***}$. Then $B$ is a unique maximally monotone extension of $A$ in $X\times X^{***}$, and  $A$ is
 of type (DV). \qede

\end{example}

\begin{proof}
By Proposition~\ref{FPDV:PR6},  $B$ is a unique maximally monotone extension of $A$ in $X\times X^{***}$.
Then applying Goldstine's theorem (see Fact~\ref{Goldst:1}), $A$ is of type (DV).
\end{proof}

\section{Structure of maximally monotone operators}\label{s:main}

We turn to the structure of maximally monotone operators in Banach
space
 whose domains have nonempty interior and we
present new and explicit structure formulas for such operators. Along the way, we provide
 new proofs of   norm-to-weak$^{*}$  closedness and of property (Q) for these operators (as recently proven by Voisei).
 Various applications and limiting examples are given.

 This section is mainly based on the work in \cite{BY1, BY2}.
\subsection{Local boundedness properties}\label{s:voi}

The next two important results now have many proofs (see also
\cite[Ch. 8]{BorVan}).

\begin{fact}[Rockafellar]\label{SubMR}\emph{(See \cite[Theorem~A]{Rock702},
 \cite[Theorem~3.2.8]{Zalinescu}, \cite[Theorem~18.7]{Si2} or \cite[Theorem~9.2.1]{BorVan}.)}
 Let $f:X\rightarrow\RX$ be a proper lower semicontinuous convex function.
Then $\partial f$ is maximally monotone.
\end{fact}

The prior result can fail in both incomplete normed spaces and in complete metrizable locally convex spaces \cite{BorVan}.

\begin{fact}[Rockafellar]\emph{(See \cite[Theorem~1]{Rock69} or  \cite[Theorem~2.28]{ph}.)}
\label{pheps:11}Let $A:X\To X^*$ be  monotone with
 $\inte\dom A\neq\varnothing$.
Then $A$ is locally bounded at $x\in\inte\dom A$,
 i.e., there exist $\delta>0$ and $K>0$ such that
\begin{align*}\sup_{y^*\in Ay}\|y^*\|\leq K,
\quad \forall y\in (x+\delta B_X)\cap \dom A.
\end{align*}
\end{fact}

Based on Fact~\ref{pheps:11}, we can develop a stronger result.

\begin{lemma}[Strong directional boundedness]\label{LemClC:1}
\emph{(See \cite[Lemma~4.1]{BY1}.)}
Let $A:X\rightrightarrows X^*$ be monotone and $x\in\inte\dom A$.
 Then there exist $\delta>0$ and  $M>0$ such that
 $x+2\delta B_X\subseteq\dom A$ and $\sup_{a\in x+2\delta B_X}\|Aa\|\leq M$.
Assume also that
$(x_0,x_0^*)$ is monotonically related to $\gra A$.
Then
 \begin{align*}\sup_{a\in\left[ x+\delta B_X,\,x_0\right[,\,
  a^*\in Aa}\|a^*\|\leq\frac{1}{\delta}\left(\|x_0-x\| +1\right)
\left(\|x^*_0\|+2M\right),
\end{align*}
where  $\left[x+\delta B_X,\,x_0\right[:=\big\{(1-t)y+tx_0\mid 0\leq t< 1, y\in x+\delta B_X\big\}$.
\end{lemma}

The following result --- originally conjectured by the first author in \cite{Bor4}  --- was established by  Voisei in \cite[Theorem~37]{VoiseiIn} as part of a more complex set of results.

\begin{theorem}[Eventual boundedness]\label{bondednet:1}\emph{(See \cite[Theorem~4.1]{BY1}.)}
Let $A:X\rightrightarrows X^*$ be  monotone
 such that $\inte\dom A\neq\varnothing$.
Then every norm $\times$ weak$^*$ convergent net in $\gra A$ is eventually bounded.
\end{theorem}

\begin{corollary}[Norm-weak$^*$ closed graph] \emph{(See \cite[Corollary~4.1]{BY1}.)}\label{CorCL:1}
Let $A:X\rightrightarrows X^*$ be maximally monotone
 such that $\inte\dom A\neq\varnothing$.
Then $\gra A$ is norm $\times$ weak$^*$ closed.
\end{corollary}

\begin{example}[Failure of graph to be norm-weak$^*$ closed]\label{ex:notclosed}
In \cite{BFG}, the authors showed that the statement of Corollary~\ref{CorCL:1}
cannot hold without the assumption of the nonempty interior domain even for the
subdifferential operators --- actually it  fails in the bw$^*$ topology. More precisely (see \cite{BFG} or
\cite[Example~21.5]{BC2011}):
 Let $f:\ell^2(\NN)\rightarrow\RX$ be defined by
 \begin{align}
 x\mapsto\max\big\{1+\langle x, e_1\rangle,\sup_{2\leq n\in\NN}
 \langle x, \sqrt{n}e_n\rangle\big\},\label{NormToEx:1}
 \end{align}
 where $e_n:=(0,\ldots,0,1,0,\cdots,0):$ the $n$th entry is $1$ and the others
are $0$.
 Then $f$ is proper lower semicontinuous and convex,
 but $\partial f$ is not norm $\times$ weak$^*$ closed.
 A more general construction in an infinite-dimensional
  Banach space $E$ is also given in \cite[Section~3]{BFG}. It is as follows:

  Let $Y$ be an infinite dimensional separable subspace of $E$, and $(v_n)_{n\in\NN}$ be
  a \emph{normalized Markushevich basis} of $Y$ with the dual coefficients $(v^*_n)_{n\in\NN}$.
  We defined $v_{p,m}$ and $v^*_{p,m}$ by
  \begin{align*}
  v_{p,m}:=\frac{1}{p}(v_p+v_{p^m})\quad\text{and}\quad
  v^*_{p,m}:=v_p^*+(p-1)v^*_{p^m},\quad m\in\NN,\, p\,\text{is prime}.
  \end{align*}
  Let $f: E\rightarrow\RX$ be defined by
 \begin{align}
 x\mapsto\iota_Y (x)+\max\big\{1+\langle x, v^*_1\rangle,\sup_{2\leq m\in\NN,\, p\,\text{is prime}}
 \langle x, v^*_{p,m}\rangle\big\}.
 \end{align}
  Then $f$ is proper lower semicontinuous and convex. We have that
$\partial f$ is not norm $\times$ bw$^*$ closed
 and hence $\partial f$ is not norm $\times$ weak$^*$ closed.
\qede
\end{example}

Let $A:X\rightrightarrows X^*$.
Following \cite{Hou}, we say $A$ has the upper-semicontinuity property \emph{property (Q)}  if
for every net $(x_{\alpha})_{\alpha\in J}$ in
$X$ such that $x_{\alpha}\longrightarrow x$, we have
\begin{align}\label{propQ}
\bigcap_{\alpha\in J}\overline{\conv\left[
\bigcup_{\beta\succeq_J\alpha} A (x_{\beta})\right]}^{\wk}
\subseteq Ax.
\end{align}

Let $A:X\To X^*$ be  monotone with $\dom A\neq\varnothing$ and consider a set $S\subseteq\dom A$.
 We define $A_S:X\To X^*$
by
\begin{align}
\gra A_{S}&=\overline{\gra A\cap(S\times X^*)}^{\|\cdot\|\times\wk}\nonumber\\
&=\big\{(x,x^*)\mid \exists\, \text{a net}\,
(x_{\alpha},x^*_{\alpha})_{\alpha\in\Gamma}\,
\text{in}\,\gra A\cap (S\times X^*)\, \text{such that}\, x_{\alpha}\longrightarrow x,
x^*_{\alpha}\weakstarly x^*\big\}\label{Deintcl}.
\end{align}
If $\inte\dom A\neq\varnothing$, we denote by
$A_{\inte}:=A_{\inte\dom A}$. We  note that $\gra A_{\dom A}=\overline{\gra A}^{\|\cdot\|\times\wk} \supseteq\gra A$ while $\gra A_S \subseteq \gra A_T$ for  $S \subseteq T$.

We now turn to consequences of these  boundedness results. The
following is the key technical proposition of this section.

\begin{proposition}\label{LemClC:3}\emph{(See \cite[Proposition~5.2]{BY1}.)}
Let $A:X\rightrightarrows X^*$ be maximally monotone
with $S \subseteq \inte\dom A\neq\varnothing$ such that $S$ is dense in $\inte\dom A$. Assume that
$x\in \dom A$ and  $v\in H_{\overline{\dom A}}(x) = \inte T_{\overline{\dom A}}(x)$.
Then there exists $x^*_0\in A_{S}(x)$ such that
\begin{align}
\sup\big\langle A_{S}(x), v\big\rangle=\big\langle x^*_0,v\big\rangle
=\sup\big\langle Ax,v\big\rangle.
\label{LETS:1}
\end{align}
In particular, $\dom A_S=\dom A$.
 \end{proposition}

\begin{corollary}\label{CorCL:4}\emph{(See \cite[Corollary~5.1]{BY1}.)}
Let $A:X\rightrightarrows X^*$ be maximally monotone with
$S \subseteq \inte\dom A\neq\varnothing$. For any  $S$  dense in $\inte\dom A$,
we have   $\overline{\conv\left[ A_S(x)\right]}^{\wk}=Ax=A_{\inte}(x), \forall x\in\inte\dom A$.
\end{corollary}

There are many possible extensions of this sort of result along the lines studied in \cite{BFK}.
Applying Proposition~\ref{LemClC:3} and Lemma~\ref{LemClC:1}, we can also quickly recapture \cite[Theorem~2.1]{Aus93}.

\begin{theorem}[Directional boundedness in Euclidean space] \label{Ausmain:1}\emph{(See \cite[Theorem~5.1]{BY1}.)}
Suppose that $X$ is finite-dimensional. Let $A:X\rightrightarrows X^*$ be maximally monotone
and $x\in\dom A$. Assume that there exist $d\in X$ and $\varepsilon_0>0$
such that $x+\varepsilon_0 d\in\inte\dom A$. Then
\[\left[Ax\right]_d:=\big\{x^*\in Ax \mid \langle x^*,d\rangle=\sup\langle Ax,d\rangle\big\}\]
is nonempty and compact. Moreover, if a sequence $(x_n)_{n\in\NN}$ in $\dom A$ is such that
$x_n\longrightarrow x$ and
\begin{align}
\lim\frac{x_n-x}{\|x_n-x\|}=d,\label{NWC:7}
\end{align}
 then for every $\varepsilon>0$, there exists $N\in\NN$ such that
 \begin{align}
 A(x_n)\subseteq \left[Ax\right]_d+\varepsilon B_{X^*},\quad \forall n\geq N.
 \label{NWC:8}
 \end{align}
\end{theorem}

\begin{theorem}[Reconstruction of $A$, I]\label{TheFom:2}\emph{(See \cite[Theorem~5.2]{BY1}.)}
Let $A:X\rightrightarrows X^*$ be maximally monotone with $S \subseteq \inte\dom A\neq\varnothing$ and  with $S$ dense in $\inte\dom A$.
Then
\begin{align}
Ax&=N_{\overline{\dom A}}(x)+\overline{\conv\left[A_{S}(x)\right]}^{\wk}
,\label{ThNp:S1}
 \quad \forall x\in X.
\end{align}

\end{theorem}

\begin{remark}{(See \cite[Remark~5.4]{BY1}.)}
If $X$ is a \emph{weak Asplund space} (as holds if $X$ has a G\^{a}teaux smooth equivalent norm, see \cite{ph,PPN,BFK}),
the nets defined in $A_{S}$  in
Proposition~\ref{LemClC:3} and Theorem~\ref{TheFom:2} can be replaced by sequences.
\qede
\end{remark}

In various classes of Banach space we can choose useful structures for $S\in S_A$, where
\begin{align*}
S_A:=\big\{
S\subseteq\inte\dom A\mid\text{$S$\, is dense in  $\inte\dom A$}\big\}.
\end{align*}
By \cite{Georg,ph,PPN,Ves1,Ves2,RockWets,BorVan}, we have multiple selections for $S$ (see below).
\begin{corollary}[Specification of $S_A$]\label{cor:cases}\emph{(See \cite[Corollary~5.2]{BY1}.)} Let $A:X\rightrightarrows X^*$ be maximally monotone
with $\inte\dom A\neq\varnothing$. We may choose the dense set $S\in S_A$ to be as follows:
\begin{enumerate}
\item\label{Cor:Ca1} In a G\^{a}teaux smooth space, entirely within the residual set of non-$\sigma$ porous points of $\dom A$,
\item\label{Cor:Ca2} In an Asplund space, to include only a subset of the generic set of points of single-valuedness and  norm to norm continuity of $A$,
\item\label{Cor:Ca3} In a separable Asplund space, to hold only  countably many angle-bounded points of $A$,
\item \label{Cor:Ca4}In  a weak Asplund space,  to include only a subset of the generic set of  points of single-valuedness (and norm to weak$^*$ continuity) of $A$,
    \item \label{Cor:Ca4b}In  a separable space,  to include  only  points of single-valuedness (and norm to weak$^*$ continuity) of $A$ whose complement is covered by a countable union of Lipschitz surfaces.
\item\label{Cor:Ca5} In finite dimensions, to use sets of full measure including only points of differentiability of $A$ (almost everywhere) \cite[Corollary~12.66(a), page~571]{RockWets}.
\end{enumerate}
\end{corollary}

These classes are sufficient but not necessary: for example, there are Asplund spaces with no equivalent G\^{a}teaux  smooth renorm \cite{BorVan}.
Note also that in \ref{Cor:Ca4b} and \ref{Cor:Ca5} we  also know that $A \diagdown S$ is a null set in the senses discussed in \cite{CheZha}.

We now restrict attention to convex functions.

\begin{corollary}[Convex subgradients]\label{CorSubd}\emph{(See \cite[Corollary~5.3]{BY1}.)}
Let $f:X\rightarrow\RX$ be proper lower semicontinuous and
 convex with $\inte\dom f\neq\varnothing$.  Let $S \subseteq \inte \dom f$ be given with $S$ dense in $\dom f$.
Then
\begin{align*}
\partial f(x)=N_{\overline{\dom f}}(x)
+\overline{\conv\left[{(\partial f)}_{S}(x)\right]}^{\wk}=N_{\dom f}(x)
+\overline{\conv\left[(\partial f)_{S}(x)\right]}^{\wk}, \quad \forall x\in X.
\end{align*}
\end{corollary}

\begin{remark}
Results closely related  to Corollary~\ref{CorSubd} have been obtained in \cite{Rock70CA,BBC1,JoTh, ThZag} and elsewhere. Interestingly, in the convex case we have obtained as much information more easily than by the direct convex analysis approach of \cite{BBC1}. \qede
\end{remark}

Now we refine Corollary \ref{CorCL:4} and Theorem \ref{TheFom:2}.

Let $A:X\rightrightarrows X^*$. We define $\widehat{A}:X\rightrightarrows X^*$ by
\begin{align}
\gra \widehat{A}:=\big\{(x,x^*)\in X\times X^*\mid x^*\in\bigcap_{\varepsilon>0}
\overline{\conv \left[A(x+\varepsilon B_X)\right]}^{\wk}
\big\}.\label{NeEL:1}
\end{align}
Clearly, we have $\overline{\gra A}^{\|\cdot\|\times \wk}\subseteq\gra\widehat{A}$.

\begin{theorem}[Reconstruction of $A$, II]\label{PropLe:2}\emph{(See \cite[Theorem~5.3]{BY1}.)}
Let $A:X\rightrightarrows X^*$ be maximally monotone
with $\inte\dom A\neq\varnothing$.
\begin{enumerate}\item \label{part:i}
 Then $\widehat{A}=A$.

In particular, $A$ has property (Q); and so has a norm $\times$  weak$^*$ closed graph.

\item  \label{part:ii} Moreover, if $S \subseteq \inte \dom A$ is dense in $\inte\dom A$ then
 \begin{align}
 \widehat{A_S}(x):&=\bigcap_{\varepsilon>0}
\overline{\conv \left[A(S \cap(x+\varepsilon B_X))\right]}^{\wk}
\supseteq\overline{\conv\left[A_{S}(x)\right]}^{\wk},\label{ThNp:S2}
 \quad \forall x\in X.
\end{align}
 Thence
  \begin{align}
Ax= \widehat{A_S}(x)+N_{\overline{\dom A}}(x)
,\label{ThNp:S3}
 \quad \forall x\in X.
\end{align}
\end{enumerate}
 \end{theorem}

\begin{remark}
Property (Q), first introduced by Cesari in Euclidean space, was recently established for maximally monotone operators with nonempty domain interior in a barreled normed space by
 Voisei  in \cite[Theorem~42]{VoiseiIn} (See also \cite[Theorem~43]{VoiseiIn} for the result under more general hypotheses.).
Several interesting characterizations of maximally monotone
operators in finite dimensional spaces, including the property (Q)
were studied by L\"{o}hne \cite{Lohne}.
\end{remark}

In general, we do not have  $Ax=\overline{\conv\left[A_{S}(x)\right]}^{\wk}, \forall x\in\dom A$, for a maximally monotone operator
$A:X\rightrightarrows X^*$ with
$S \subseteq \inte\dom A\neq\varnothing$ such that $S$ is dense in $\dom A$.

 We give a simple example to demonstrate this.

\begin{example}(See \cite[Example~6.1]{BY1}.)
Let $C$ be a closed convex subset of $X$ with $S\subseteq\inte C\neq\varnothing$ such that
$S$ is dense in $C$.
Then $N_C$ is maximally monotone  and $\gra (N_C)_{S}=C\times \{0\}$, but $N_C(x)\neq\overline{\conv \left[(N_C)_{S}(x)\right]}^{\wk},
\forall x\in \bd C$.  We thus must have
$\bigcap_{\varepsilon>0}
\overline{\conv \left[N_C(x+\varepsilon B_X)\right]}^{\wk}=N_C(x),\,\forall x\in X$. \qede
\end{example}

While the subdifferential operators in Example~\ref{ex:notclosed} necessarily
fail to have property (Q), it is possible for operators with no points of continuity to possess the property. Considering any closed linear mapping $A$ from a reflexive space $X$ to its dual,
we have $\widehat{A}=A$ and hence $A$ has property (Q). More generally:

\begin{example}(See \cite[Example~6.2]{BY1}.)
Suppose that $X$ is reflexive. Let $A:X\rightrightarrows X^*$ be such that
 $\gra A$ is nonempty closed and convex.
Then $\widehat{A}=A$ and hence $A$ has property (Q). \qede
\end{example}

It would be interesting to know whether $\widehat A$ and $A$ can differ for a maximal operator with norm $\times$ weak$^*$ closed graph.

Finally, we illustrate what Corollary \ref{CorSubd} says in the case of $x\mapsto \iota_{B_X}(x)+\frac{1}{p}\|x\|^p.$

\begin{example}\label{SForEx:1}(See \cite[Example~6.3]{BY1}.)
Let  $p> 1$ and $f:X\rightarrow\RX$ be defined by
\begin{align*}
x\mapsto \iota_{B_X}(x)+\frac{1}{p}\|x\|^p.
\end{align*}
Then for every $x\in \dom f$, we have
\begin{align*}
N_{\dom f}(x)&=\begin{cases}\RR_+\cdot Jx,\, &\text{if}\, \|x\|=1;\\
\{0\},\, &\text{if}\, \|x\|<1
\end{cases},\quad
(\partial f)_{\inte}(x)=\begin{cases}\|x\|^{p-2}\cdot Jx,\, &\text{if}\, \|x\|\neq0;\\
\{0\},\, &\text{otherwise}
\end{cases}
\end{align*}
where $J:=\partial\tfrac{1}{2}\|\cdot\|^2$ and $\RR_+:=\left[0,+\infty\right[$.
Moreover, $\partial f=N_{\dom f}+(\partial f)_{\inte}=N_{\dom f}+\partial \frac{1}{p}\|\cdot\|^p$,
and then $\partial f(x)\neq (\partial f)_{\inte}(x)=
\overline{\conv\left[(\partial f)_{\inte}(x)\right]}^{\wk}, \forall x\in\bd\dom f$.
We also have $\bigcap_{\varepsilon>0}
\overline{\conv \left[\partial f(x+\varepsilon B_X)\right]}^{\wk}=\partial f(x),\,\forall x\in X$. \qede
\end{example}

\subsection{Characterizations of the domain and range of $A$}\label{ssec:conv}
The following is the classical result on the convexity of the closure of the domain of a maximally monotone operator with nonempty interior domain.
\begin{fact}[Rockafellar]
\emph{(See \cite[Theorem~1]{Rock69} or \cite[Theorem~27.1 and Theorem~27.3]{Si2}.)}
\label{f:referee02c}
Let $A:X\To X^*$ be  maximal monotone with $\inte\dom A\neq\varnothing$. Then
$\inte\dom A=\inte\overline{\dom A}$ and $\overline{\dom A}$ is convex.
\end{fact}

Let $A:X\To X^*$ be   monotone.  We say $A$ is \emph{rectangular} if $\dom A\times \ran A
\subseteq\dom F_A$.  Now we note the following interesting result on the characterization of
 the sum of the ranges of two monotone operators.
\begin{fact}[Reich]\emph{(See \cite[Theorem~2.2]{Reich}, or \cite[Corollary~31.6]{Si2}.)}\label{RREctan} Suppose that $X$ is reflexive.
Let $A, B:X\To X^*$ be   monotone such that $A+B$ is maximally monotone. If either
$A$ and $B$ are rectangular, or $\dom A\subseteq \dom B$ and $B$ is rectangular,
then the Brezis-Haraux condition
\begin{align*}
\inte\ran (A+B)=\inte(\ran A+\ran B)\quad\text{and}
\quad \overline{\ran (A+B)}=\overline{\ran A+\ran B}.
\end{align*}
holds.
\end{fact}
In the setting of a Hilbert space, Brezis and Haraux proved the above result in
\cite[Theorem~3, pp.~173--174]{BreHara}.

The strong result below  follows directly from the definition of
operators of type (BR):

\begin{proposition}\label{ProBR2}\emph{(See \cite[Proposition~3.5]{BY2}.)}
Let $A:X\rightrightarrows X^*$ be maximally monotone
 and $(x,x^*)\in X\times X^*$. Assume that $A$
is of type (BR) and that $\inf_{(a,a^*)\in\gra A}\langle
x-a,x^*-a^*\rangle>-\infty$. Then $x\in\overline{\dom A}$ and
$x^*\in\overline{\ran A}$. In particular, \[\overline{\dom A}=\overline{P_X\left[\dom F_A\right]}\quad \mbox{and}
 \quad \overline{\ran A}=\overline{P_{X^*}\left[\dom F_A\right]}.\] In particular, $\overline{\dom A}$ and $\overline{\ran A}$ are both  convex.
\end{proposition}

We recall that every monotone operator of type (FPV) has a convex
closure of its domain,  while every maximally monotone continuous linear
operator is of type (FPV) (see \cite[Theorem~46.1]{Si2} or
\cite{BY2}). But as  Remark~\ref{RProBR:2} shows, a
maximally monotone bounded linear operator  need not be of type (BR).

We turn to an interesting related result on the domain of $A$.
\begin{theorem}\label{ProdLm:3}\emph{(See \cite[Theorem~3.6]{BY2}.)}
Let $A:X\rightrightarrows X^*$ be maximally monotone.
Then
\begin{align*}
\overline{\conv\left[\dom A\right]}=\overline{P_X\left[\dom F_A\right]}.
\end{align*}
\end{theorem}

\begin{remark}
Theorem~\ref{ProdLm:3}
provides an affirmative answer to a question
 posed by  Simons in
 \cite[Problem~28.3, page~112]{Si2}.
\end{remark}

Following the lines of the proof of \cite[Theorem~3.6]{BY2}, we
obtain the following  counterpart result.

\begin{theorem}\label{ProdLm:4}
Let $A:X\rightrightarrows X^*$ be maximally monotone.
Then
\begin{align*}
\overline{\conv\left[\ran A\right]}^{\wk}=\overline{P_{X^*}\left[\dom F_A\right]}^{\wk}.
\end{align*}
\end{theorem}
\begin{proof}
By Fact~\ref{f:Fitz}, it suffices to show that
\begin{align}
P_{X^*}\left[\dom F_A\right]\subseteq \overline{\conv\left[\ran A\right]}^{\wk}.
\label{Lsee:10}
\end{align}
Let $(z,z^*)\in\dom F_A$.  We shall show that
\begin{align}
z^*\in\overline{\conv\left[\ran A\right]}^{\wk}.
\label{Lsee:11}
\end{align}
Suppose to the contrary that
\begin{align}
z^*\notin\overline{\conv\left[\ran A\right]}^{\wk}.
\label{Lsee:12}
\end{align}
Since $(z,z^*)\in\dom F_A$,  there exists $r\in\RR$ such that
\begin{align}
F_A(z,z^*)\leq r.\label{Lsee:14a}
\end{align}
By the Separation theorem, there exist $\delta>0$ and  $y_0\in X$ with $\|y_0\|=1$
such that
\begin{align}
\langle y_0, z^*-a^*\rangle>\delta, \quad \forall a^*\in \conv\left[\ran A\right].
\label{Lsee:14}
\end{align}
Let $n\in\NN$.
Since $z^*\notin\overline{\conv\left[\ran A\right]}^{\wk}$, $(z+ny_0, z^*)\notin\gra A$.
By the maximal monotonicity of $A$, there exists $(a_n, a_n^*)\in\gra A$ such that
\begin{align}
&\langle z-a_n, a^*_n-z^*\rangle>\langle n y_0, z^*-a^*_n\rangle
\quad \Rightarrow\langle z-a_n, a^*_n-z^*\rangle>n\delta\quad\text{(by \eqref{Lsee:14})}\nonumber\\
&\quad \Rightarrow \langle z-a_n, a^*_n\rangle+\langle z^*,a_n\rangle>n\delta+\langle z, z^*\rangle.
\label{Lsee:15}
\end{align}
Then we have
\begin{align*}
F_A(z,z^*)&\geq \sup_{n\in\NN}\big\{\langle z-a_n, a^*_n\rangle+\langle z^*,a_n\rangle\big\}
\geq \sup_{n\in\NN}\big\{n\delta +\langle z, z^*\rangle\big\}
=+\infty,\end{align*}
which contradicts \eqref{Lsee:14a}.
Hence $z\in \overline{\conv\left[\ran A\right]}^{\wk}$ and in consequence \eqref{Lsee:10} holds.
\end{proof}
\begin{remark} In Theorem~\ref{ProdLm:4},  we cannot replace the w$^*$ closure by the norm closure. For Example,
let $A$ be defined as in Theorem~\ref{PBABD:2}. Theorem~\ref{ProdLm:4} implies that $\overline{\ran A}^{\wk}=\overline{P_{X^*}\left[\dom F_A\right]}^{\wk}$, however, $\overline{P_{X^*}\left[\dom F_A\right]}\nsubseteq
\overline{\ran A}$ by Remark~\ref{RProBR:1}.

More concrete  example is as follows.
Let $A_{\alpha}$ be define as in Example~\ref{FPEX:1}.
By Remark~\ref{RProBR:2} and Theorem~\ref{ProdLm:4},
$\overline{\ran A_{\alpha}}^{\wk}=\overline{P_{X^*}\left[\dom F_{A_{\alpha}}\right]}^{\wk}$ but  $\overline{P_{X^*}\left[\dom F_{A_{\alpha}}\right]}\nsubseteq
\overline{\ran A_{\alpha}}$.
\end{remark}

\section{Results on linear relations}\label{s:linear}
This section is mainly based on the work in \cite{BBWY3, BWY7, BBYE} by
Bauschke, Borwein, Burachik, Wang and Yao. During the 1970s Brezis
and Browder presented a now classical characterization of maximal
monotonicity of monotone linear relations in reflexive spaces.

\begin{theorem}[Brezis-Browder in reflexive Banach space \cite{BB76,Brezis-Browder}]
 \label{thm:brbr} Suppose that $X$ is reflexive.
Let $A\colon X \To X^*$ be a monotone linear relation
such that $\gra A$ is closed.
Then
$A$ is maximally monotone if and only if the adjoint
$A^*$ is monotone.
\end{theorem}
We extend this result in the setting of a general real Banach space.
(See also \cite{Si3} and \cite{Si7} for Simons' recent extensions in
the context of
{symmetrically self-dual Banach} (SSDB)
spaces as defined in \cite[\S21]{Si2} and of Banach SNL spaces.)

\begin{theorem}[Brezis-Browder in general Banach space] \label{thm:bbwy}
\emph{(See \cite[Theorem~4.1]{BBWY5}.)}
Let $A\colon X\rightrightarrows X^*$ be a monotone linear
 relation such that $\gra A$ is closed.
Then  $A$ is maximally monotone of type (D) if
and only if  $A^*$ is monotone.
\end{theorem}

This  also gives an affirmative answer to a question of Phelps and Simons
 \cite[Section~9, item~2]{PheSim}:
\begin{quotation}
\emph{Let $A:\dom A\to X^*$ be linear and maximally monotone.
 Assume that $A^*$ is monotone. Is $A$ necessarily
 of type (D)}?
 \end{quotation}

Recently,  Stephen Simons strengthens Theorem~\ref{thm:bbwy} in \cite{Si7}:

\begin{theorem}[Simons]
\emph{(See \cite[Corollary~6.6]{Si7}.)}
Let $A\colon X\rightrightarrows X^*$ be a monotone linear
 relation such that $\gra A$ is closed.
 Then
the following are equivalent.
\begin{enumerate}
\item  $A$ is maximally monotone of type (D).
\item
$A^*$ is monotone.
\item $A^*$ is maximally monotone with respect to
$X^{**}\times X^*$.
\end{enumerate}
\end{theorem}
 \medskip

We give a corresponding but  negative answer to a question posed in \cite[Chapter 3.5, page~56]{YaoPhD} (see Example~\ref{Exphts:1} below).
 \begin{quote}
\emph{Let $A\colon X \To X^*$ be a monotone linear relation
such that $\gra A$ is closed. Assume $A^*|_X$ is monotone.
Is $A$ necessarily maximally monotone?}
\end{quote}

 If $Z$ is a real  Banach space with dual $Z^*$ and a set $S\subseteq
Z$, we define $S^\bot$ by $S^\bot := \{z^*\in Z^*\mid\langle
z^*,s\rangle= 0,\quad \forall s\in S\}$. Given a subset $D$ of
$Z^*$, we define $D_{\bot}$ \cite{PheSim} by $D_\bot := \{z\in
Z\mid\langle z,d^*\rangle= 0,\quad \forall d^*\in D\}$.

 \begin{example}\label{Exphts:1}
 Let $X$ be nonreflexive, and $e\in X^{**}\backslash{X}$.
 Let $A:X\rightrightarrows X^*$ by $\gra A:=\{0\}\times e_{\bot}$.
 Then $A$ is a monotone linear relation with closed graph, and $\gra A^*=\spand\{e\}\times X^*$.
 Moreover, $A^*|_X$ is monotone but $A$ is not maximally monotone.\qede
 \end{example}

 \begin{proof}
 Clearly, $A$ is a monotone linear relation and $\gra A$ is closed.
 Since $e\notin X$, $e\neq 0$. Thus, $ e_{\bot}\neq X^*$ and hence
 $A$ is not maximally monotone.

 Let $(z^{**}, z^*)\in X^{**}\times X^*$.  Then we have
 \begin{align*}
 (z^{**}, z^*)\in\gra A^{**}
 &\Leftrightarrow \langle z^{**}, a^*\rangle+\langle z^*, -0\rangle=0,\quad
 \forall a^*\in e_\bot\Leftrightarrow \langle z^{**}, a^*\rangle=0,\quad
 \forall a^*\in e_\bot\\
 &\Leftrightarrow z^{**}\in (e_\bot)^{\bot}\\
 &\Leftrightarrow z^{**}\in \spand\{e\}\quad\text{(by  \cite[Proposition~2.6.6(c)]{Megg})}.
 \end{align*}
 Hence $\gra A^*=\spand\{e\}\times X^*$ and then
 $\gra (A^*|_X)=\{0\}\times X^*$.
 Thus, $A^*|_X$ is monotone.
  \end{proof}

\begin{remark}
Example~\ref{Exphts:1} gives a negative answer to a question posed in \cite[Chapter 3.5, page~56]{YaoPhD}. Note that by \cite[Proposition~5.4(iv)]{BBWY5} or \cite[Proposition~3.2.10(iii), page~25]{YaoPhD}, the converse of \cite[Chapter 3.5, page~56]{YaoPhD} is true, that is,
\begin{quote}
\emph{Let $A\colon X \To X^*$ be a maximally monotone linear relation.
Then $A^*|_X$ is monotone.}
\end{quote}
\end{remark}

\begin{fact}{\rm (See \cite[Propositions~3.5, 3.6 and 3.7 and Lemma~3.18]{BWY7}.)}
\label{FE:1}
Suppose that  $X=\ell^2$,  and that
$A:\ell^2\rightrightarrows \ell^2$ is given by
\begin{align}Ax:=\frac{\bigg(\sum_{i< n}x_{i}-\sum_{i> n}x_{i}\bigg)_{n\in\NN}}{2}
=\bigg(\sum_{i< n}x_{i}+\tfrac{1}{2}x_n\bigg)_{n\in\NN},
\quad \forall x=(x_n)_{n\in\NN}\in\dom A,\label{EL:1}\end{align}
where $\dom A:=\Big\{ x:=(x_n)_{n\in\NN}\in \ell^{2}\mid \sum_{i\geq 1}x_{i}=0,
 \bigg(\sum_{i\leq n}x_{i}\bigg)_{n\in\NN}\in\ell^2\Big\}$ and $\sum_{i<1}x_{i}:=0$.
Then
\begin{align}
\label{PF:a2}
A^*x= \bigg(\thalb x_n + \sum_{i> n}x_{i}\bigg)_{n\in\NN},
\end{align}
where
\begin{equation*}
x=(x_n)_{n\in\NN}\in\dom A^*=\bigg\{ x=(x_n)_{n\in\NN}\in \ell^{2}\;\; \bigg|\;\;
 \bigg(\sum_{i> n}x_{i}\bigg)_{n\in\NN}\in \ell^{2}\bigg\}.
\end{equation*}
Then $A$ provides an at most single-valued linear relation
such that the following hold.
\begin{enumerate}
 \item
 $A$\label{NEC:1} is maximally monotone and skew.
\item\label{NEC:2} $A^*$ is maximally monotone but not skew.
\item\label{NEC:4} $F_{A^*}^{*}(x^*,x)=F_{A^*}(x,x^*)
=\iota_{\gra A^*}(x,x^*)+\scal{x}{x^*},\quad \forall(x,x^*)\in X\times X$.
\item\label{NEC:5} $\langle A^*x, x\rangle=\tfrac{1}{2}s^2,
\quad \forall x=(x_n)_{n\in\NN}\in\dom A^*\ \text{with}\quad
s:=\sum_{i\geq1} x_i$.
\end{enumerate}
\end{fact}

Let $F:X\times X^*\rightarrow\RX$, and define $\pos F$ \cite{Si2}
 by  \begin{align*}\pos F:=\big\{(x,x^*)\in X\times X^*\mid F(x,x^*)=\langle x,x^*\rangle\big\}.
 \end{align*}

The following  result due to Simons generalizes
 the result of Br\'{e}zis, Crandall and Pazy \cite{BreCA}.

\begin{fact}[Simons]
\emph{(See \cite[Theorem~34.3]{Si2}.)}\label{BCPCon}
Suppose that $X$ is reflexive.
Let $F_1,F_2:X\times X^*\rightarrow\RX$ be
 proper lower semicontinuous and convex functions
  with $P_X\dom F_1\cap P_X\dom F_2\neq\varnothing$.
Assume that $F_1, F_2$ are BC--functions and  that there exists
an increasing function $j:\left[0,+\infty\right[\rightarrow
\left[0,+\infty\right[$ such that the implication
\begin{align*}&(x,x^*)\in\pos F_1, (y,y^*)\in\pos F_2, x\neq y\ \text{and}\
\langle x-y, y^*\rangle=\|x-y\|\cdot\|y^*\|\\
&\quad\Rightarrow\|y^*\|\leq
j\big(\|x\|+\|x^*+y^*\|+\|y\|+\|x-y\|\cdot\|y^*\|\big)
\end{align*}
holds. Then
$M:=\big\{(x, x^*+y^*)\mid (x,x^*)\in\pos F_1, (x,y^*)\in\pos F_2\big\}$
is a maximally monotone set.
\end{fact}

\begin{example}\label{FCPEX:3}{(See \cite[Example~5.2]{BBWY4}.)}
Suppose that $X$ and $A$ are as in Fact~\ref{FE:1}.
Set $e_1:=(1,0,\ldots,0,\ldots)$, i.e., there is a $1$
 in the first place and all  others entries are $0$, and $C:=\left[0,e_1\right]$.
 Let $j:\left[0,+\infty\right[
 \rightarrow\left[0,+\infty\right[$ be an increasing function such that
 $j(\gamma)\geq\tfrac{\gamma}{2}$ for every $\gamma\in\left[0,+\infty\right[$.
Then the following hold.
\begin{enumerate}
\item\label{BCPONA:E01} $F_{A^*}$ and $F_{N_C}=\iota_C\oplus\sigma_C$ are BC--functions.
\item\label{BCPONA:E1}
$(F_{A^*}\Box_2 F_{N_C})(x,x^*)=
\begin{cases}
\langle x,A^*x\rangle + \sigma_C(x^*-A^*x),&\text{if $x\in C$;}\\
\pinf, &\text{otherwise,}
\end{cases}
\quad \forall (x,x^*)\in X\times X^*$.
\item\label{BCREX:E2} Then
\begin{align*}
F^*_{A^*}(x^*,0)+F_{N_{C}}^*(A^*e_1-x^*,0)>
(F_{A^*}\Box_2 F_{N_{C}})^*(A^*e_1,0),\quad\forall x^*\in X.
\end{align*}
\item\label{BCREX:E3}
The implication
\begin{align*}&(x,x^*)\in\pos F_{N_{C}} , (y,y^*)\in\pos F_{A^*}, x\neq y\ \text{and}\
\langle x-y, y^*\rangle=\|x-y\|\cdot\|y^*\|\nonumber\\
&\quad\Rightarrow\|y^*\|\leq\tfrac{1}{2}\|y\|\leq
j\big(\|x\|+\|x^*+y^*\|+\|y\|+\|x-y\|\cdot\|y^*\|\big)
\end{align*}
 holds.
\item\label{BCREX:E4} $A^*+N_C$ is maximally monotone. \qede
\end{enumerate}
\end{example}

 Example~\ref{FCPEX:3} shows  the following conjecture fails in
general (see \cite{BBWY4} for more on Open Problem~\ref{prob2}.

\begin{openprob}\label{prob2} \emph{Suppose that $X$ is reflexive. Let $F_1,F_2:X\times X^*\rightarrow\RX$ be
 proper lower semicontinuous and convex functions
  with $P_X\dom F_1\cap P_X\dom F_2\neq\varnothing$.
Assume that $F_1, F_2$ are BC--functions and  that there exists
an increasing function $j:\left[0,+\infty\right[\rightarrow
\left[0,+\infty\right[$ such that the implication
\begin{align*}&(x,x^*)\in\pos F_1, (y,y^*)\in\pos F_2, x\neq y\ \text{and}\
\langle x-y, y^*\rangle=\|x-y\|\cdot\|y^*\|\\
&\quad\Rightarrow\|y^*\|\leq
j\big(\|x\|+\|x^*+y^*\|+\|y\|+\|x-y\|\cdot\|y^*\|\big)
\end{align*}
holds.  Then, is it true that, for all $(z,z^*)\in X\times X^*$, there exists
$v^*\in X^*$ such that
\begin{align}
\label{Probcon}F_1^*(v^*,z)+F_2^*(z^*-v^*,z)\leq (F_1\Box_2 F_2)^*(z^*,z)?
\end{align}}
\end{openprob}

Finally, we provide some results on the partial inf-convolution of
two Fitzpatrick functions associated with maximally monotone
operators, which has important consequences  for the ``sum problem"
(see the discussion in Section~\ref{s:openp}).

\begin{proposition}\label{F12}\emph{(See \cite[Proposition~7.1.11, page~164]{YaoPhD}.)}
Let $A,B\colon X\To X^*$  be maximally monotone and
suppose that $\bigcup_{\lambda>0} \lambda\left[\dom A-\dom B\right]$ is a closed subspace of $X$.
Then
$ F_A\Box_2F_B$ is proper, norm$\times$weak$^*$ lower semicontinuous and convex,
and the partial infimal convolution is exact everywhere.
\end{proposition}

Theorem~\ref{FS6} below was  proved  in \cite[Theorem~5.10]{BWY3} for  a reflexive space.
It can be extended  to a general Banach space.
\begin{theorem}[Fitzpatrick function of the sum]\label{FS6}\emph{(See \cite[Theorem~5.2]{BBYE}.)}
Let $A,B\colon X\To X^*$ be maximally monotone linear relations,
and suppose that $\dom A-\dom B$ is closed.
Then $$F_{A+B}= F_A\Box_2F_B,$$
and the partial infimal convolution is exact everywhere.
\end{theorem}

Theorem~\ref{FS6} provides a new approach to showing the maximal
monotonicity of two maximally monotone linear relations (see
\cite[Theorem~5.5]{BBYE}),
 which was  first used by  Voisei in \cite{Voisei06} while  Simons gave another proof in \cite[Theorem~46.3]{Si2}.

\begin{theorem}\label{tf:main}\emph{(See \cite[Theorem~5.5]{BBYE}.)}
Let $A:X\To X^*$ be a maximally monotone linear relation.
Suppose $C$ is a nonempty closed convex subset of $X$,
and  that $\dom A \cap \inte C\neq \varnothing$.
Then $F_{A+N_C}= F_A\Box_2F_{N_C}$,
and the partial infimal convolution is exact everywhere.
\end{theorem}

\section{Open problems in  Monotone Operator Theory}\label{s:openp}

As discussed in \cite{Bor1,Bor2,Bor3,BorVan}, the two most central
open questions in monotone operator theory  in a general real Banach
space are almost certainly the following:

Let $A, B\colon  X\To X^*$ be maximally monotone.
\begin{enumerate}\item[(i)] The ``sum problem":
Assume that   $A,B$ satisfy
\emph{Rockafellar's constraint qualification}, i.e., $\dom A\cap\inte\dom B\neq\varnothing$
\cite{Rock70}.  \emph{Is the sum operator $A+B$ necessarily
maximally monotone}, which is so called the ``sum problem"?
 \item[(ii)]  \emph{Is $\overline{\dom A}$ necessarily convex?}
Rockafellar showed that it is true  for every operator with nonempty interior domain \cite{Rock69} and as we saw in Section \ref{ssec:conv} it is now known to
hold for most classes of maximally monotone operators (see also\cite[Section~44]{Si2}).
\end{enumerate}

A positive answer to various restricted versions of (i) implies a
positive answer to (ii) \cite{BorVan,Si2}. Some recent developments on the sum problem can be found in  Simons' monograph
\cite{Si2} and \cite{Bor1,Bor2,Bor3,BorVan,BY3, ZalVoi, MarSva5, VV2,Yao3,Yao2}.
In \cite{BY3}, we showed if the following conjecture is true then
 the  sum problem would have an affirmative answer.
\begin{quote}\emph{\textbf{Conjecture}
Let $A:X\To X^*$ be a maximally monotone linear relation,
and let $B: X\rightrightarrows X^*$ be maximally monotone
such that $\bigcup_{\lambda>0} \lambda\left[\dom A-\dom B\right]=X$.
Then $A+B$ is maximally monotone.}
\end{quote}
In \cite{BY3}, we showed the above conjecture is true when $A$ and $B$ satisfy  Rockafellar's constraint qualification:
$\dom A\cap\inte\dom B\neq\varnothing$.
At the end of this section, we will list some interesting open problems on the special cases of the sum problem.
  Simons showed that the closure of the domain of every (FPV) operator is convex \cite[Theorem~44.2]{Si2}.
  However, we do not know if every maximally monotone is of type (FPV).
   Recent progress regarding (ii) can be
found in \cite{BY2}.

In the following, we show that one possible approach to the sum
problem cannot be feasible. By  \cite[Lemma~23.9]{Si2} or
\cite[Proposition~4.2]{BM}, $F_A\Box_2 F_{B}\geq F_{A+B}$. It
naturally raises a question: Does  equality  always hold
 under the Rockafellar's constraint qualification?
If this were true, then it would directly solve the sum problem in
the affirmative (see \cite{Voi1, Si2} and \cite[Chapter~7]{YaoPhD}).
However,  in general, it cannot hold. The easiest example probably
is \cite[Example~4.7]{BM}  by Bauschke, McLaren and Sendov using two projection operators
  on one dimensional space.

   Here we give another counterexample of a
a maximally monotone linear relation  and the subdifferential of
  a proper lower semicontinuous sublinear function, which thus also implies that
 we cannot establish the maximality of the sum of a linear relation $A$ and the subdifferential of
   a proper lower semicontinuous sublinear function $f$ by showing that
$F_A\Box_2 F_{\partial f}= F_{A+\partial f}$ always holds.

\begin{example}\label{Example:main} {(See \cite[Example~7.1.14, page~167]{YaoPhD}.)}
Let $X$ be a Hilbert space, $B_X$ be the closed unit ball of $X$ and $\Id$ be the
identity mapping\index{identity mapping} from $X$ to $X$.
Let $f:x\in X\rightarrow \|x\|$. Then we have
\begin{align}
F_{\partial f}\Box_2 F_{\Id}(x,x^*)=\|x\|+\begin{cases}0,\;&\text{if}\,\|x+x^*\|\leq1;\\
\tfrac{1}{4}\|x+x^*\|^2-\tfrac{1}{2}\|x+x^*\|+\tfrac{1}{4},\;&\text{if}\,\|x+x^*\|>1.\label{esee:7}
\end{cases}\end{align}
We also have $F_{\partial f+\Id}\neq F_{\partial f}\Box_2 F_{\Id}$ when $X=\RR$. \qede
\end{example}

Now we show that another possible approach to the sum
problem cannot be feasible either.

\begin{quote}
Let $F:X\times X^*\rightarrow\RX$ be proper lower semicontinuous and convex. Assume that
\begin{align*}
F(x, x^*)\geq\langle x,x^*\rangle,\quad F^*(x^*,x)\geq\langle x,x^*\rangle,\quad
\forall (x,x^*)\in X\times X^*.
\end{align*}
Is $\pos F$ is a maximally monotone set?
\end{quote}

If the above conjecture were true, then the sum problem would have an affirmative answer by setting $F:=F_A\Box_2 F_B$. Burachik and Svaiter showed the conjecture holds when $X$ is reflexive
(see \cite[Theorem~3.1]{BurSVi} or \cite[Theorem~1.4(b)]{SiZ}). We give the following example to show that it cannot be true in a general Banach space.

\begin{example}\label{Exphts:2}
 Let $X$ be nonreflexive, and $e\in X^{**}\backslash{X}$.
 Let $F: X\times X^*\rightarrow\RX$ be defined by $F:=\iota_{\{0\}\times e_{\bot}}$.
 Then $F^*=\iota_{X^*\times\spand\{e\}}$ on $X^*\times X^{**}$, and \begin{align*}F(x, x^*)\geq\langle x,x^*\rangle,\quad F^*(x^*,x)\geq\langle x,x^*\rangle,\quad
\forall (x,x^*)\in X\times X^*.\end{align*}
However, $\pos F=\{0\}\times e_{\bot}$  is not a maximally monotone set.\qede
 \end{example}
\begin{proof}We have $F$ is proper lower semicontinuous and convex.
Similar to the proof of  Example~\ref{Exphts:1}, we have  $F^*=\sigma_{\{0\}\times e_{\bot}}=\iota_{X^*\times\spand\{e\}}$.
Then we have
$F\geq\langle \cdot,\cdot\rangle,\quad F^{*\intercal}\geq\langle \cdot,\cdot\rangle$
on $X\times X^*$.
Clearly, $\pos F=\{0\}\times e_{\bot}$. By Example~\ref{Exphts:1}, it is  not a maximally monotone set.
\end{proof}

\begin{remark}[Conjecture]
Finally, we conjecture that every nonreflexive space admits an
operator that is not of type (BR) and so also not of type
(D).\end{remark}

We also list some interesting open problems on special cases of the sum problem:

\begin{problem}
Let $A:X\rightarrow X^*$ be a continuous monotone linear operator,
and let $B:X\rightrightarrows X^*$ be maximally monotone.
  Is
$A+B$ necessarily maximally monotone$?$
\end{problem}

\begin{problem}
 Let $f:X\rightarrow \RX$ be a proper lower semicontinuous convex function,
and let $B:X\rightrightarrows X^*$ be maximally monotone with $\dom \partial f\cap\inte\dom B\neq\varnothing$.
  Is
$\partial f +B$ necessarily maximally monotone$?$
\end{problem}

\begin{problem}
 Let  $A:X\rightrightarrows X^*$ be maximally monotone with convex domain.
  Is
$A$ necessarily of type (FPV)?
\end{problem}

Let us recall a problem posed by S. Simons in \cite[Problem~41.2]{Si}

\begin{problem}
Let $A:X\rightrightarrows X^*$ be of type (FPV),
let $C$ be a nonempty closed convex subset of $X$,
and suppose that $\dom A \cap \inte C\neq \varnothing$.
Is $A+N_C$ necessarily maximally monotone$?$

\end{problem}

A more general problem:

\begin{problem}
Let $A, B:X\To X^*$ be  maximally monotone
with $\dom A\cap\inte\dom B\neq\varnothing$. Assume
that $A$ is of type (FPV).
Is $A+B$ necessarily maximally monotone$?$

\end{problem}

\paragraph{Acknowledgments} The authors thank Dr.~Heinz Bauschke and Dr.~Andrew Eberhard for their
careful reading and pertinent comments on various parts of this
paper. The authors also thank the editors and the anonymous  referee for their
 pertinent and constructive comments.
Jonathan  Borwein and Liangjin Yao were partially supported
by various Australian Research Council grants.


\end{document}